\documentclass{amsart}
\usepackage{amsmath}
\usepackage{paralist}
\usepackage{amsfonts}
\usepackage{amssymb}
\usepackage{amsthm}
\usepackage{amscd}
\usepackage{mathtools}
\usepackage{amsrefs}
\usepackage{float}
\usepackage{tikz}
\usepackage{graphicx}
\usepackage[colorlinks=true]{hyperref}
\hypersetup{urlcolor=blue, citecolor=red}
\usepackage{hyperref}

\newtheorem{theorem}{Theorem}[section]
\newtheorem{corollary}[theorem]{Corollary}

\newtheorem{lemma}[theorem]{Lemma}
\newtheorem{proposition}[theorem]{Proposition}

\theoremstyle{definition}

\newtheorem{remark}[theorem]{Remark}

\newcommand{\T}{\mathbb{T}}
\newcommand{\R}{\mathbb{R}}
\newcommand{\Z}{\mathbb{Z}}

\begin{document}

\title[Higher dimensional Benjamin-Ono and Zakharov-Kuznetsov equations] 
      {On the Cauchy problem for higher dimensional Benjamin-Ono and Zakharov-Kuznetsov equations}

\author[R. Schippa]{Robert Schippa}
\address{Fakult\"at f\"ur Mathematik, Universit\"at Bielefeld, Postfach 10 01 31, 33501 Bielefeld, Germany}
\email{robert.schippa@uni-bielefeld.de}

\begin{abstract}
A family of dispersive equations is considered which links a higher-dimensional Benjamin-Ono equation and the Zakharov-Kuznetsov equation. For these fractional Zakharov-Kuznetsov equations new well-posedness results are proved using transversality and localization of time to small frequency dependent time intervals.
\end{abstract}

\thanks{Financial support by the German Science Foundation (IRTG 2235) is gratefully acknowledged.}
\maketitle

\section{Introduction}
In this note well-posedness of the higher-dimensional fractional Zakharov-Kuznetsov equations
\begin{equation}
\label{eq:fractionalBenjaminOnoEquation}
\left\{\begin{array}{cl}
\partial_t u + \partial_{x_1} (-\Delta)^{a/2} u &= u \partial_{x_1} u, \quad (t,x) \in \R \times \mathbb{K}^n, \quad 1 \leq a \leq 2 \\
u(0) &= u_0 \in H^s(\mathbb{K}^n), \end{array} \right.
\end{equation}
is discussed, where $n \geq 2$ and $ \mathbb{K} \in \{ \mathbb{R}, \mathbb{T} \}$.\\
In the one-dimensional case \eqref{eq:fractionalBenjaminOnoEquation} becomes the Benjamin-Ono equation (cf. \cite{Benjamin1974}, see e.g. \cite{Saut2018} for a recent survey) for $a=1$ and the Korteweg-De Vries equation (see \cite{KillipVisan2018} for the sharp global well-posednesss result) for $a=2$. In the one-dimensional case the equations are best understood and extensively studied.\\
In higher dimensions \eqref{eq:fractionalBenjaminOnoEquation} yields a generalization of the Benjamin-Ono equation for $a=1$ (cf. \cite{LinaresRianoRogersWright2019,Maris2002,PelinovskyShrira1995}) and for $a=2$ the Zakharov-Kuznetsov equation (cf. \cite{ZakharovKuznetsov1974}) is recovered.\\
By local well-posedness we mean that for any $u_0 \in H^s$ there is $T=T(\Vert u_0 \Vert_{H^s})$ such that $S_T^\infty: H^\infty \rightarrow C([0,T],H^\infty)$ extends uniquely to a continuous mapping $S_T^s: H^s \rightarrow C([0,T],H^s)$.\\
The energy method \cite{BonaSmith1975} yields well-posedness for $s>\frac{n+2}{2}$, but neglects the dispersive properties. These are clearly stronger in Euclidean space than in the fully periodic case. We discuss solutions in Euclidean space first, for which we can show stronger well-posedness results consequently.\\
Already in the one-dimensional case it is well-known that the data-to-solution mapping for dispersion coefficients $1 \leq a <2 $ is not uniformly continuous (cf. \cite{KochTzvetkov2005,HerrIonescuKenigKoch2010,MolinetSautTzvetkov2001}).\\
Also, in two-dimensions it was proved for $a=1$ in \cite{LinaresRianoRogersWright2019} that the data-to-solution mapping is not $C^2$. Local well-posedness was proved for $a=1$ provided that $s>5/3$ in \cite{LinaresRianoRogersWright2019} using shorttime linear Strichartz estimates (see also \cite{KochTzvetkov2003}).\\
Here, we improve the local well-posedness for $n=2$ and interpolate between $a=1$ and $a=2$ to recover in the limiting case the currently best local well-posedness for the Zakharov-Kuznetsov equation $s>1/2$ (cf. \cite{GruenrockHerr2014,MolinetPilod2015}) in two dimensions and $s>1$ in three dimensions \cite{MolinetPilod2015,RibaudVento2012}. The results in higher dimensions seem to be new for $1 < a \leq 2$.\\
Here, we use transversality and localization of time to small frequency dependent time intervals (cf. \cite{IonescuKenigTataru2008,rsc2018BilinearStrichartzEstimates}) to prove the following theorem:
\begin{theorem}
\label{thm:LocalWellposednessEuclideanSpace}
Let $\mathbb{K} = \mathbb{R}$, $1 \leq a < 2$ and $s>\frac{n+3}{2}-a$. Then \eqref{eq:fractionalBenjaminOnoEquation} is locally well-posed. 
\end{theorem}
We sketch the method of proof. Let $N \in 2^{\mathbb{N}_0}$ denote a dyadic number and $P_N$ the inhomogeneous Littlewood-Paley projector, i.e.,
\begin{equation*}
(P_N f) \widehat (\xi) = \begin{cases}
\Phi(\xi) \hat{f}(\xi), \quad N=1 \\
\chi_N(\xi) \hat{f}(\xi), \quad \text{else,}
\end{cases}
\end{equation*}
where $\Phi, \chi_N \in C^\infty_c$, $supp(\Phi) \subseteq B(0,2)$, $supp(\chi_N) \subseteq B(0,2N) \backslash B(0,N/2)$ and $\Phi + \sum_N \chi_N \equiv 1$.\\
Further, let $S_a(t)$ denote the linear propagator of \eqref{eq:fractionalBenjaminOnoEquation}, that is
\begin{equation*}
\widehat { S_a(t) u_0 } (\xi) = e^{-i t \xi_1 |\xi|^a} \hat{u}_0(\xi)
\end{equation*}
The most problematic interaction happens in case a low frequency interacts with a high frequency because the derivative nonlinearity
\begin{equation*}
\partial_{x_1} (P_N u P_K u) \quad (K \ll N)
\end{equation*}
possibly requires one to recover a whole derivative. The derivative loss is partially ameliorated by the following bilinear Strichartz estimate:
\begin{proposition}
\label{prop:BilinearStrichartzEstimateHighLow}
Let $n \geq 2, \; K,N \in 2^{\mathbb{N}_0}, \; K \ll N$. Then, we find the following estimate to hold:
\begin{equation}
\label{eq:BilinearStrichartzEstimateHighLow}
\Vert P_N S_a(t) u_0 P_K S_a(t) v_0 \Vert_{L^2_{t,x}(\R \times \R^n)} \lesssim \left( \frac{K^{n-1}}{N^a} \right)^{1/2} \Vert P_N u_0 \Vert_{L^2} \Vert P_K v_0 \Vert_{L^2}
\end{equation} 
\end{proposition}
This proposition is an easy consequence of general transversality considerations (cf. \cite{Bourgain1998RefinementsStrichartzInequality}).\\
Apparently, this is still insufficient to recover the derivative loss for $1 \leq a <2$. To overcome the gap we additionally localize time in a frequency dependent way (cf. \cite{IonescuKenigTataru2008}).\\
In the following we motivate at which frequency dependent time localization we can treat the most problematic $High \times Low \rightarrow High$-interaction utilizing \eqref{eq:BilinearStrichartzEstimateHighLow}. For $K \ll N$ one finds
\begin{equation}
\label{eq:ShorttimeBilinearAmelioration}
\begin{split}
&\Vert \partial_{x_1} (P_N S_a(t) u_0 P_K S_a(t) v_0) \Vert_{L^1([0,T];L^2(\R^n))} \\
&\lesssim N |T|^{1/2} \Vert P_N S_a(t) u_0 P_K S_a(t) v_0 \Vert_{L^2([0,T]; L_x^2(\R^n)} \\
&\lesssim |T|^{1/2} N \left( \frac{K^{n-1}}{N^a} \right)^{1/2} \Vert P_N u_0 \Vert_{L^2(\R^n)} \Vert P_K v_0 \Vert_{L_x^2(\R^n)}
\end{split}
\end{equation}
This suggests that for $T(N)=N^{2-a}$ this peculiar interaction can be estimated for $s>(n-1)/2$, which will be carried out in Section \ref{section:NonlinearEstimates}.
In the one-dimensional case this had been done for dispersion generalized Benjamin-Ono equations (cf. \cite{Guo2012DispersionGeneralizedBenjaminOno,GuoPengWangWang2011}).\\
This argument will be sufficient to handle $High \times Low \rightarrow High$-interactions and $High \times High \rightarrow High$-interactions for $n=2$. For $High \times High \rightarrow High$-interactions at $n \geq 3$ linear Strichartz estimates (cf. \cite{LinaresRianoRogersWright2019}) are used:
\begin{proposition}
\label{prop:LinearStrichartzEstimates}
Let $n \geq 3$, $1 \leq a \leq 2$ and $2 \leq p,q \leq \infty$, $p \neq \infty$. Then, we find the following estimate to hold
\begin{equation}
\label{eq:LinearStrichartzEstimates}
\begin{split}
\Vert S_a(t) f \Vert_{L_t^q(\R,L_x^p(\R^n))} &\lesssim \Vert f \Vert_{\dot{H}^s(\R^n)} \\
\Vert S_a(t) f \Vert_{L_t^q([0,T],L_x^p(\R^n))} &\lesssim_{T} \Vert f \Vert_{H^s(\R^n)}
\end{split}
\end{equation}
provided that $\frac{2}{q} + \frac{2}{p} = 1$ and $s=n \left( \frac{1}{2} - \frac{1}{p} \right) - \frac{a+1}{q}$.
\end{proposition}
Since localization in time erases the dependence on the initial data, one still has to carry out energy estimates, which will give a worse regularity threshold to close the argument, namely $s>\frac{n+3}{2}-a$. This will be done in Section \ref{section:EnergyEstimates}.\\
We will use a variant of the function spaces from \cite{IonescuKenigTataru2008,ChristHolmerTataru2012} to prove a priori estimates in the first step, next, $L^2$-Lipschitz dependence for initial data of higher regularity is discussed and finally, continuous dependence is proved by the Bona-Smith argument (cf. \cite{BonaSmith1975}).\\
The strategy of the proof closely follows the arguments from \cite{rsc2018BilinearStrichartzEstimates} where the argument was applied to periodic solutions.\\
The approach from \cite{rsc2018BilinearStrichartzEstimates} does not apply directly to periodic solutions because this would require the dispersion relation to split
\begin{equation*}
\varphi(\xi) = \sum_{i=1}^n \chi(\xi_i)
\end{equation*}
This is true for another possible generalization of the (fractional) Benjamin-Ono equation
\begin{equation}
\label{eq:multiDirectionalBenjaminOnoEquation}
\partial_t u + \sum_{i=1}^n \partial_{x_i} |D_{x_i}|^{a} u = \sum_{i=1}^n \partial_{x_i} (u^2)/2, \quad (t,x) \in \R \times \mathbb{K}^n,
\end{equation}
where $\mathbb{K} \in  \{\R; \T\}$.\\
Here, for $n=2$, $a=2$ we recover a Cauchy problem which is equivalent to the Zakharov-Kuznetsov equation (cf. \cite{BenArtziKochSaut2003,GruenrockHerr2014}).\\
Another Benjamin-Ono-Zakharov-Kuznetsov equation was considered in \cite{RibaudVento2017}:
\begin{equation}
\label{eq:fractionalBenjaminOnoEquationII}
\left\{\begin{array}{cl}
\partial_t u - \partial_{x_1} D_{x_1}^a u + \partial_{x_1} \partial_{x_2}^2 u &= u \partial_{x_1} u, \quad (t,x) \in \R \times \R^2, \quad 1 \leq a \leq 2 \\
u(0) &= u_0 \in H^s(\R^n), \end{array} \right.
\end{equation}
Here, only dispersion in the $x_1$-component was decreased. Local and global well-posedness results for \eqref{eq:fractionalBenjaminOnoEquationII} were also proved via frequency dependent time localization.\\
Lastly, we remark that the local well-posedness result from Theorem \ref{thm:LocalWellposednessEuclideanSpace} gives global well-posedness in the energy space $H^{a/2}(\R^2)$ for sufficiently large $a$ in the two-dimensional case due to conservation of energy
\begin{equation*}
E(u) = \int_{\R^n} |D^{a/2} u |^2 - \frac{1}{3} u^3(t,x) dx
\end{equation*}
Another conserved quantity is the mass
\begin{equation*}
M(u) = \int_{\R^n} u^2(t,x) dx,
\end{equation*}
but a well-posedness result in $L^2$ seems to be far beyond the methods of this paper.\\
Thus, iteration of Theorem \ref{thm:LocalWellposednessEuclideanSpace} for $s=a/2$ yields:
\begin{corollary}
Let $n=2, \mathbb{K} = \mathbb{R}$ and $a>5/3$. Then, \eqref{eq:fractionalBenjaminOnoEquation} is globally well-posed for $s=a/2$.
\end{corollary}
We turn to a discussion of the fully periodic case. In the two-dimensional case the anisotropic Sobolev space $H^{s,0}(\T^2)$ is also considered for the Zakharov-Kuznetsov equation, where
\begin{equation*}
\Vert f \Vert^2_{H^{s,0}(\T^2)} = \sum_{(\xi,\eta) \in \mathbb{Z}^2} (1+|\xi|^2)^{s} |\hat{f}(\xi,\eta)|^2
\end{equation*}
In previous works (\cite{LinaresPantheeRobertTzvetkov2018,BustamanteJimenezUrreaMejia2019}) local well-posedness has only been considered in isotropic Sobolev spaces, but since this is a larger space we also consider well-posednesss for these initial data. We prove the following theorem:
\begin{theorem}
\label{thm:LocalWellposednessPeriodicCase}
Let $\mathbb{K} = \mathbb{T}$, $n=2$ and $1 \leq a \leq 2$ or $n \geq 2$, $a=2$ and $s>(n+1)/2$. Then \eqref{eq:fractionalBenjaminOnoEquation} is locally well-posed in $H^s(\T^n)$ and for $n=2$, $a=2$ \eqref{eq:fractionalBenjaminOnoEquation} is locally well-posed in $H^{s,0}(\T^2)$.
\end{theorem}
In case $n=2$ this improves the results from \cite{LinaresPantheeRobertTzvetkov2018,BustamanteJimenezUrreaMejia2019} where local well-posedness was proved in $H^s(\T^2)$ for $a \in \{1,2\}$ provided that $s>5/3$ for $a=2$ and $s>7/4$ for $a=1$. In these works shorttime linear Strichartz estimates were used. In the present work this result is modestly improved by transversality considerations and corresponding results are proven in higher dimensions. However, the covered regularities are still far from the energy space. To make further progress one probably needs a better comprehension of the resonance set which appears to be more delicate than for the Kadomtsev Petviashvili-equations (cf. \cite{Bourgain1993KPII,Zhang2016}).\\
The strategy of proof is the same as for solutions on Euclidean space: In suitable function spaces we will prove a priori estimates for solutions, Lipschitz continuous dependence for differences of solutions in $L^2$ for higher regular initial data and finally continuous dependence in $H^s$ by the Bona-Smith approximation. The conclusion of the proof is similar to the Euclidean case.\\
Key ingredient will be bilinear convolution estimates for the space-time Fourier transform of functions which will be localized in frequency and modulation. These will be derived in Subsection \ref{subsection:PeriodicBilinearEstimates}. Here, the transversality considerations from Euclidean space will again come into play. However, we always have to localize time reciprocally to the highest involved frequency so that transversality becomes observable. Therefore, we can not lower the regularity at which our method of proof yields local well-posedness as the dispersion coefficients increase compared to the Euclidean case.\\
After deriving these bilinear convolution estimates the argument follows the real line case. Thus, for the sake of clarity of presentation the details after the derivation of the bilinear convolution estimates will only be presented for the two-dimensional Zakharov-Kuznetsov equation.
\section{Linear Strichartz estimates}
This section is devoted to the proof of Proposition \ref{prop:LinearStrichartzEstimates} which was carried out for $a=1$ in \cite{LinaresRianoRogersWright2019}. The required modifications are easy, but the proof is contained for the sake of completeness. We start with a dispersive estimate:
\begin{proposition}
Let $a \geq 1$, $n\geq 3$ and $\psi: \R^n \rightarrow \R$ be a smooth radial function supported in $B_n(0,2) \backslash B_n(0,1/2)$. Then, we find the following estimate to hold:
\begin{equation}
\label{eq:kernelEstimate}
\left| \int \psi(|\xi|) e^{i(t \xi_1 |\xi|^a + x. \xi)} d\xi \right| \leq C |t|^{-1} 
\end{equation}
with $C$ only depending on $n$, $\psi$ and $a$.
\end{proposition}
\begin{proof}
We rewrite the integral in spherical coordinates to find
\begin{equation*}
\begin{split}
I(x,t) &= \int_0^\infty dr \underbrace{r^{n-1} \psi(r)}_{\rho(r)} \int_{\mathbb{S}^{n-1}} d\sigma(\omega) e^{it(r^{a+1} \omega_1 + x_1 r \omega_1 + \ldots + x_n r \omega)} \\
&= \int_0^\infty \rho(r) \hat{\sigma}(y_{x,t}(r)) dr,
\end{split}
\end{equation*}
where $y_{x,t}(r) = (t r^{a+1} +x_1 r, x_2 r, \ldots, x_n r)$.\\
Recall the decay 
\[
|\hat{\sigma}(y)| \lesssim (1+|y|)^{-\frac{n-1}{2}}
\]
This is already enough to prove the claim for $n \geq 4$.\\
Indeed, partition $supp(f) = E_1 \cup E_2$, where $E_1 = \{ r \in supp(\rho) | |t r^{a+1} + x_1 r| \leq 1 \}$ and $|E_1| \lesssim |t|^{-1}$. To see this note that $|t r^{a+1} + x_1 r| \leq 1$ implies $|tr^a + x_1| \leq 2$ and, by change of variables,
\begin{equation*}
\int_{1/2}^2 1_{ \{ |t r^a + x_1| \leq 2 \} }(r) \rho(r) dr = \int_{r^\prime \sim 1} 1_{\{ |t r^\prime + x_1| \leq 2 \} } \rho(r^\prime) dr^\prime \leq C |t|^{-1}
\end{equation*}
where $C$ depends on $\psi$, $n$ and $a$.\\
Similarly, $E_2 \subseteq \{ r \in supp(\rho) | |t r^a + x_1| \geq 2 \}$ and consequently,
\begin{equation*}
\begin{split}
\int_{E_2} \rho(r) | \hat{\sigma}(y_{x,t}(r))| dr &\leq \int_{|tr^a + x_1| \geq 2}  \rho(r) |tr^{a+1} + x_1 r|^{- \frac{n-1}{2}} dr \\
&\leq C \int_{|t r + x_1| \geq 2} |tr + x_1|^{-\frac{n-1}{2}} dr =C |t|^{-\frac{n-1}{2}} \int_{|r+x_1/t| \geq 2/|t|} |r+x_1/t|^{-\frac{n-1}{2}} dr \\
\end{split}
\end{equation*}
and after linear change of variables we estimate by $ C|t|^{-1}$.\\
We turn to $n=3$. Here, we make use of the asymptotic expansion
\begin{equation*}
\hat{\sigma}(y) = c \frac{e^{i \Vert y \Vert}}{\Vert y \Vert} + c \frac{e^{-i \Vert y \Vert}}{\Vert y \Vert} + \mathcal{E}_{x,t}(y),
\end{equation*}
where $|\mathcal{E}_{x,t}(y)| \lesssim \Vert y \Vert^{-2} \quad ( \Vert y \Vert \gg 1)$.\\
Set $\phi(r) = \sqrt{f(r)}$, where $f(r) = (tr^{a+1} + x_1 r)^2 + r^2 \Vert x^\prime \Vert^2$ and 
\begin{align*}
F^1 &= \{ r \in supp(\rho) | |tr^{a+1} + x_1 r| \leq 1 \} \cap \{ r \in supp(\rho) | |f^\prime(r)| \leq |t| \} \supseteq E^1, \\
F^2 &= \{ r \in supp(\rho) | |tr^{a+1} + x_1 r| \geq 1, \quad |f^\prime(r)| \geq |t| \} \subseteq E^2
\end{align*}
Below, we see that $|F^1| \lesssim |t|^{-1}$, which means that this contribution is controlled by $|\hat{\sigma}| \lesssim 1$.\\
Moreover, the contribution of $\mathcal{E}_{x,t}$ when integrating over $F^2$ is controlled by the higher dimensional argument due to $F^2 \subseteq E^2$ and sufficient decay to run the above argument.\\
A computation yields
\begin{equation*}
\begin{split}
f^\prime(r) &= 2t^2 (a+1) r (r^a - r_{-})(r^a-r_{+}),\\
r_{\pm} &= -\frac{(a+2) x_1}{2(a+1) t} \pm \sqrt{\left( \frac{a+2}{a+1} \right)^2 \left( \frac{x_1}{t} \right)^2 - \frac{x_1^2}{(a+1) t^2} - \frac{\Vert x^\prime \Vert^2}{(a+1) t^2}}
\end{split}
\end{equation*}
We can suppose that $\frac{x_1}{t} \sim 1$ and $\frac{\Vert x^\prime \Vert^2}{t^2} \ll 1$, since otherwise $|f(r)| \gtrsim |t|$, so that the roots are real and separated.\\
In fact, $|r_{\pm}| \sim 1$ and $|r_+ - r_{-}| \sim 1$. Moreover, whenever $f^\prime$ vanishes, then $|f^{\prime \prime}|$ is still bounded away from zero and thus, $|F^1| \lesssim |t|^{-1}$.\\
For the contribution of $e^{i \Vert y \Vert}/ \Vert y \Vert$ over $F^2$ note that we can write
\begin{equation*}
\int \frac{e^{i \phi(r)}}{\phi(r)} \rho(r) dr \sim \int \frac{d}{dr} [ e^{i \phi(r)} ] \frac{\rho(r)}{f^\prime(r)} dr
\end{equation*}
Next, the domain of integration is divided into a finite union of intervals, where $\rho/f^\prime$ is monotone. On each such interval integration by parts yields the desired result.
\end{proof}
\begin{remark}
The dispersive estimate follows also from \cite[Proposition~4.7.,~p.~260]{KochTataru2005}.
\end{remark}
From the dispersive estimate Strichartz estimates are derived by standard arguments.
\begin{proof}[Proof of Proposition \ref{prop:LinearStrichartzEstimates}]
For $n \geq 3$ the dispersive estimate and conservation of mass give by interpolation
\begin{equation*}
\Vert S_a(t) P_1 f \Vert_{L^p(\R^n)} \lesssim \Vert \tilde{P}_1 f \Vert_{L^{p^\prime}(\R^n)} \quad (2 \leq p \leq \infty)
\end{equation*}
and combination with the $TT^*$-argument (cf. \cite{Tomas1975,GinibreVelo1979,KeelTao1998}) proves Strichartz estimates
\begin{equation*}
\Vert S_a(t) P_1 f \Vert_{L_t^q(\R,L_x^p(\R^n))} \lesssim \Vert \tilde{P}_1 f \Vert_{L^2(\R^n)}
\end{equation*}
provided that $\frac{2}{q} + \frac{2}{p} = 1$, $p \neq \infty$. A scaling argument gives for $p,q$ like above
\begin{equation*}
\Vert S_a(t) P_N f \Vert_{L_t^q(\R,L_x^p(\R^n))} \lesssim N^s \Vert \tilde{P}_N f \Vert_{L^2(\R^n)}, \quad s = n \left( \frac{1}{2} - \frac{1}{p} \right) - \frac{a+1}{q}
\end{equation*}
and \eqref{eq:LinearStrichartzEstimates} follows from Littlewood-Paley theory.
\end{proof}
\section{Bilinear Strichartz estimates}
\label{section:BilinearStrichartzEstimates}
Purpose of this section is to prove bilinear Strichartz estimates as stated in Proposition \ref{prop:BilinearStrichartzEstimateHighLow}. Whereat the proof is straight-forward in case of separated frequencies, it requires more care to treat the $High \times High \times High$-interaction
\begin{equation}
\label{eq:HighHighHighInteraction}
\int \int_{\R^2 \times [0,T]} P_{N_1} S_a(t) u_0 P_{N_2} S_a(t) v_0 P_{N_3} S_a(t) w_0 dx dy dt, \quad N \sim N_1 \sim N_2 \sim N_3,
\end{equation}
where we shall see that it is still amenable to a bilinear Strichartz estimate.\\
Both cases follow from the following more general well-known transversality estimate:
\begin{proposition}
\label{prop:EuclideanTransversality}
Let $U_i$ be open sets in $\mathbb{R}^n$, $\varphi_i \in C^1(U_i,\mathbb{R})$ and let $u_i$ have Fourier support in balls of radius $r$ which are contained in $U_i$ for $i=1,2$. Moreover, suppose that $|\nabla \varphi(\xi_1) - \nabla \varphi(\xi_2)| \geq N > 0 $, whenever $\xi_1 \in U_1$, $\xi_2 \in U_2$.\\
Then, we find the following estimate to hold:
\begin{equation}
\label{eq:EuclideanTransversality}
\Vert e^{it \varphi(\nabla/i)} u_1 e^{it \varphi(\nabla/i)} u_2 \Vert_{L^2_{t,x}(\mathbb{R} \times \mathbb{R}^n)} \lesssim_n \frac{r^{\frac{n-1}{2}}}{N^{1/2}}  \Vert u_1 \Vert_{L^2(\mathbb{R}^n)} \Vert u_2 \Vert_{L^2(\mathbb{R}^n)}
\end{equation}
\end{proposition}
In order to apply Proposition \ref{prop:EuclideanTransversality} we have to analyze the group velocity $v_a(\xi)=-\nabla \varphi_a(\xi)$, where $\varphi_a(\xi) = \xi_1 |\xi|^a$.\\
We have
\begin{equation}
\label{eq:groupVelocities}
\partial_1 \varphi_a(\xi) = |\xi|^a + a \xi_1^2 |\xi|^{a-2}, \quad \partial_2 \varphi_a(\xi) = a \xi_1 \xi_2 |\xi|^{a-2}
\end{equation}
\begin{proof}[Proof of Proposition \ref{prop:BilinearStrichartzEstimateHighLow}]
First, divide $B_{2N} \backslash B_{N/2}$ into finitely overlapping balls of radius $K$, which we denote by the family $(R_L)$. Then, from almost orthogonality
\begin{equation}
\label{eq:bilinearDecomposition}
\Vert P_N S(t) u_0 P_K S(t) v_0 \Vert^2_{L^2_{t,x}} \lesssim \sum_L \Vert R_L S(t) u_0 P_K S(t) v_0 \Vert^2_{L^2_{t,x}}
\end{equation}
To estimate the terms from the sum we use Proposition \ref{prop:EuclideanTransversality}. From \eqref{eq:groupVelocities} we find $|\partial_1 \varphi(\xi)| \geq (N/2)^a$ for $|\xi| \geq N/2$ and $|\partial_2 \varphi_a(\xi)| \leq (1+a)(2K)^a$ and \eqref{eq:EuclideanTransversality} implies
\begin{equation*}
\eqref{eq:bilinearDecomposition} \lesssim \sum_L \left( \frac{K}{N} \right) \Vert R_L u_0 \Vert^2_{L^2} \Vert P_K v_0 \Vert_{L^2}^2 = \left( \frac{K}{N^a} \right)^{1/2} \Vert P_N u_0 \Vert^2_{L^2} \Vert P_K v_0 \Vert^2_{L^2},
\end{equation*}
which completes the proof.
\end{proof}
Next, we turn to the case of three comparable frequencies in the plane as depicted in \eqref{eq:HighHighHighInteraction}. We prove the following proposition:
\begin{proposition}
\label{prop:HighHighHighInteraction}
Let $N \gg 1$ and suppose that $\xi_i \in \R^2, N/8 \leq |\xi_i| \leq 8N$ for $i=1,2,3$ and $\xi_1+\xi_2+\xi_3=0$. Then, there are $i,j \in \{1,2,3\}$ with
\begin{equation*}
|v_a(\xi_i) - v_a(\xi_j)| \gtrsim N^a
\end{equation*} 
\end{proposition}
\begin{proof}
A key observation is that for $|\xi_2| \leq c |\xi|$ or $|\xi_1| \leq c |\xi|$, where $c$ is a small constant, a Taylor expansion of $|\xi|$ around the large component reveals
\begin{equation*}
\begin{split}
\partial_1 \varphi_a(\xi) &= (1+a)|\xi_1|^a + O(\xi_2^2 |\xi_1|^{a-2}) \\
&= (1+a) |\xi_1|^a + O(c^2 |\xi_1|^a) \quad \quad (|\xi_2| \leq c|\xi|) \\
\partial_1 \varphi_a(\xi) &= |\xi_2|^a + O(c^2 |\xi_2|^a) \quad \quad (|\xi_1| \leq c |\xi|)
\end{split}
\end{equation*}
This means that as soon as one component dominates the other one, the propagation into $x_1$-direction is essentially governed by the group velocity associated to a (fractional) one-dimensional Benjamin-Ono equation, which has been considered in \cite{rsc2018BilinearStrichartzEstimates}.\\
To deal with different sizes of the components for $\xi \in \R^2$ we introduce the notation $\xi \in ( A,B )$, where $A,B \in \{Low,Medium,High\}$ and $\xi \in (X,Y)$, where $X,Y \in \{+,-\}$ to indicate $\xi_1 \geq 0, \xi_2 \leq 0$. E.g. $\xi \in (High(+),Medium(-))$ means $|\xi_1| \geq \frac{c |\xi|}{2}$, $|\xi_2| \in [ c^3 |\xi|, \frac{c|\xi|}{2}]$, $\xi_1 \geq 0, \xi_2 \leq 0$ or $\xi \in (Low, High(-))$ means $|\xi_1| \leq c^3 |\xi|$, $|\xi_2| \geq \frac{c |\xi|}{2}$, $\xi_2 \leq 0$.\\
Here, $c$ is a small dimensional constant chosen so that the error terms in the above Taylor expansion can be neglected in the following considerations.\\
We sort the frequencies according to the above system.\\
Suppose that the components of any frequency are all at least of medium size, so that no component of the three frequencies is low.\\
Then, by \eqref{eq:groupVelocities} $|\partial_2 \varphi_a(\xi)| \geq c^5 |\xi|^a$ for $i=1,2,3$. Next, observe that for $\xi_i \in (+,+)$ or $ \xi_i \in (-,-)$ we have $\partial_2 \varphi(\xi_i) \geq c^5 |\xi|^a$ and in case of mixed signs $\xi_i \in (+,-)$ or $\xi_i \in (-,+)$ we have $\partial_2 \varphi_a(\xi_i) \leq - c^{5} |\xi|^a$, and the estimate $|\partial_2 \varphi_a(\xi_i) - \partial_2 \varphi_a(\xi_j)| \geq c^5 |\xi|^a$ is immediate.\\
Next, we turn to the case where all components have size greater than $ c^3 |\xi|$ and all frequencies are of equal signs (the case of mixed signs will be analogous).\\
Say $\xi_1 \in (High(+),Medium(+))$, $\xi_2 \in (High(+),High(+))$, $\xi_3 \in (High(-),High(-))$.\\
Write $\xi_{21} = \alpha \xi_{11}$, $\xi_{22} = \beta \xi_{12}$, where $\alpha, \beta \in [c^5,c^{-5}]$ and it follows
\begin{equation*}
\begin{split}
&|\partial_2 \varphi_a(\xi_1) - \partial_2 \varphi_a(\xi_3)| \\
&= \left| \frac{a \xi_{11} \xi_{12}}{(\xi_{11}^2 + \xi_{12}^2)^{\frac{2-a}{2}}} - \frac{a (1+ \alpha) \xi_{11} (1+\beta) \xi_{12}}{((1+\alpha)^2 \xi_{11}^2 + (1+\beta)^2 \xi_{12}^2)^{\frac{2-a}{2}}} \right| \\
&\geq c^5 a \frac{|\xi_{11} \xi_{12}|}{(\xi_{11}^2 + \xi_{12}^2)^{\frac{2-a}{2}}} \gtrsim |\xi|^a
\end{split}
\end{equation*}
Next, we suppose that there is one low component involved, say $\xi_1 \in (Low, High)$. Suppose that there is a frequency $\xi_j \in (High,High)$. Then, we find $|\partial_2 \varphi_a(\xi_{1})| = O(c^3 |\xi|^a)$ and $|\partial_2 \varphi_a(\xi_j)| \gtrsim c^2 |\xi|^a$, hence $|\partial_2 \varphi_a(\xi_1) - \partial_2 \varphi_a(\xi_j)| \gtrsim c^2 |\xi|^a$ which yields the desired transversality.\\
With $|\xi_{12}| \sim |\xi|$ there is another frequency, say $\xi_2$ with $|\xi_{22}| \sim |\xi|$ and by the above consideration suppose next that $\xi_2 \in (Low,High)$ or $\xi_2 \in (Medium, High)$.\\
Either way, $|\xi_{31}| \leq |\xi_{11}|+|\xi_{12}| \leq c |\xi_{11}|$ and we can expand $\partial_1 \varphi(\xi_i)$ in the second component of the frequencies to find that the analysis reduces to the one-dimensional fractional Benjamin-Ono equation and hence, there are $\xi_i$ and $\xi_j$ with
\begin{equation*}
|\partial_1 \varphi_a(\xi_i) - \partial_1 \varphi_a(\xi_j)| \gtrsim |\xi|^a
\end{equation*}
The same argument applies in case $\xi_1 \in (High,Low)$. In case there is $\xi_j \in (High,High)$ the difference satisfies $|\partial_2 \varphi_a(\xi_1) - \partial_2 \varphi_a(\xi_j)| \gtrsim c^2 |\xi|^a$ and in case there is no $\xi_j \in (High,High)$ we can expand in the first frequency component to reduce the analysis to the one-dimensional fractional Benjamin-Ono equation according to which there are $\xi_i,\xi_j$ such that $|\partial_1 \varphi_a(\xi_i) - \partial_1 \varphi_a(\xi_j)| \gtrsim |\xi|^a$.\\
The proof is complete.
\end{proof}

\section{Function spaces}
In this section we discuss the shorttime function spaces which are used to prove the local well-posedness results. The iteration scheme is the same for solutions in Euclidean space and for fully periodic solutions. However, in Euclidean space we do not have to use Fourier transform in time which allows for a simplification of the construction compared to the periodic case.\\
Shorttime $L^2$-valued $U^p$-/$V^p$-spaces will be utilized like in \cite{ChristHolmerTataru2012,rsc2018BilinearStrichartzEstimates}. Here, we will be very brief and instead refer to these works for a presentation of the basic function space properties. The notation will be the same like in the aforementioned works. For a careful exposition see \cite{HadacHerrKoch2009,HadacHerrKoch2009Erratum}. The $V^p$-spaces are the usual function spaces containing functions of bounded $p$-variation and the $U^p$-spaces are atomic spaces which are the respective predual spaces. Roughly, $U^2$ serves as a substitute for $H^{1/2}$, which does not embed into $L^\infty$, but any $U^p$-function is bounded.\\
The $U^p$-/$V^p$-spaces are adapted to free solutions in the usual way:
\begin{align*}
\Vert u \Vert_{U^p_a(I;L^2)} &= \Vert S_a(-t) u(t) \Vert_{U^p(I;L^2)} \\
\Vert v \Vert_{V^p_a(I;L^2)} &= \Vert S_a(-t) v(t) \Vert_{V^p(I;L^2)} \\
\Vert w \Vert_{DU^2_{a}(I;L^2)} &= \Vert S_a(-t) w(t) \Vert_{DU^2_{a}(I;L^2)}
\end{align*}
Motivated by \eqref{eq:ShorttimeBilinearAmelioration} we choose $T(N) = N^{a-2}$ as frequency dependent time localization.\\
Below we shall only deal with the case $1 \leq a < 2$, since for $a=2$ the localization to small frequency dependent time intervals is no longer necessary and the analysis comes down to the Fourier restriction analysis without localization in time from \cite{GruenrockHerr2014}.\\
Letting $\chi_I$ denote a sharp cut-off to a time interval $I$ the shorttime $U^2$-space into which the solution to \eqref{eq:fractionalBenjaminOnoEquation} will be placed is given by
\begin{equation*}
\Vert u \Vert^2_{F^s_a(T)} = \sum_{N \geq 1} N^{2s} \sup_{\substack{|I|=N^{a-2},\\ I \subseteq [0,T]}} \Vert P_N \chi_I u \Vert^2_{U^2_a(I;L^2)}
\end{equation*}
The corresponding space for the nonlinearity is defined by
\begin{equation*}
\Vert f \Vert^2_{N^s_a(T)} = \sum_{N \geq 1} N^{2s} \sup_{\substack{|I|=N^{a-2},\\ I \subseteq [0,T]}} \Vert P_N \chi_I u \Vert^2_{DU^2_{a}(I;L^2)}
\end{equation*}
and the energy space is
\begin{equation*}
\Vert u \Vert^2_{E^s(T)} = \sum_{N \geq 1} N^{2s} \sup_{t \in [0,T]} \Vert P_N u(t) \Vert_{L^2}
\end{equation*}
The shorttime norm of a smooth solution to \eqref{eq:fractionalBenjaminOnoEquation} is propagated as follows:
\begin{equation*}
\Vert u \Vert_{F^s_a(T)} \lesssim \Vert u \Vert_{E^s(T)} + \Vert \partial_{x_1} (u^2) \Vert_{N^s_a(T)}
\end{equation*}
(cf. \cite[Lemma~3.8,~p.~12]{rsc2018BilinearStrichartzEstimates}).\\
Moreover, since $U^p_a$-atoms are piecewise free solutions estimates for free solutions extend to $U^p_a$-functions.
\begin{proposition}
Let $n \geq 3$, $1 \leq a \leq 2$, $N \in 2^{\mathbb{N}_0}$ and $I$ be an interval. Suppose that $2/q+2/p=1$, $2 \leq q,p < \infty$. Then, we find the following estimate to hold:
\begin{equation}
\Vert P_N u(t) \Vert_{L_t^q(I;L_x^p(\R^n))} \lesssim N^s \Vert P_N u_0 \Vert_{U_a^q(I;L^2)},
\end{equation}
where $s=n \left( \frac{1}{2} - \frac{1}{p} \right) - \frac{a+1}{q}$.
\end{proposition}
This also remains valid for bilinear estimates.
\begin{proposition}
Let $1 \leq a < 2$, $N_1 \gg N_2 $ and $I$ be an interval with $|I|=N_1^{a-2}$. Then, we find the following estimates to hold:
\begin{align}
\label{eq:bilinearU2Estimate}
\Vert P_{N_1} u_1 P_{N_2} u_2 \Vert_{L^2_{t,x} (I \times \R^n)} &\lesssim \left( \frac{N_2^{n-1}}{N_1^{a}} \right)^{1/2} \Vert P_{N_1} u_1 \Vert_{U^2_a(I)} \Vert P_{N_2} u_2 \Vert_{U^2_a(I)} \\
\label{eq:bilinearV2Estimate}
\Vert P_{N_1} u_1 P_{N_2} u_2 \Vert_{L^2_{t,x} (I \times \R^n)} &\lesssim \left( \frac{N_2^{n-1}}{N_1^a} \right)^{1/2} \log \langle N_1 \rangle^2 \Vert P_{N_1} u_1 \Vert_{V^2_a(I)} \Vert P_{N_2} u_2 \Vert_{V_a^2(I)}
\end{align}
\end{proposition}
\begin{proof}
\eqref{eq:bilinearU2Estimate} is immediate from atomic decompositions (cf. \cite[Proposition~2.19,~p.~929]{HadacHerrKoch2009}) and \eqref{eq:bilinearV2Estimate} follows from an interpolation argument (cf. \cite[Proposition~2.20,~p.~930]{HadacHerrKoch2009}).
\end{proof}
\section{Nonlinear estimates}
\label{section:NonlinearEstimates}
This section is devoted to the propagation of the nonlinearity in the shorttime function spaces.
\begin{proposition}
Let $1 \leq a \leq 2$, $n \geq 2$, $s>(n-1)/2$. Then, we find the following estimates to hold:
\begin{align}
\label{eq:NonlinearEstimateI}
\Vert \partial_x (u v) \Vert_{N_a^{s}(T)} &\lesssim \Vert u \Vert_{F_a^s(T)} \Vert v \Vert_{F_a^s(T)} \\
\label{eq:NonlinearEstimateII}
\Vert \partial_x (u v) \Vert_{N_a^{0}(T)} &\lesssim \Vert u \Vert_{F_a^{0}(T)} \Vert v \Vert_{F_a^s(T)}
\end{align}
\end{proposition}
\begin{proof}
After using Littlewood-Paley theory we are reduced to the analysis of $High \times Low \rightarrow High$-, $High \times High \rightarrow High$- and $High \times High \rightarrow Low$-interaction. Carrying out the summation in the shorttime function spaces gives \eqref{eq:NonlinearEstimateI} and \eqref{eq:NonlinearEstimateII}.\\
Suppose that $N_3 \sim N_1 \gg N_2$. Then, we compute
\begin{equation*}
\begin{split}
\Vert P_{N_3} \partial_{x_1} (P_{N_1} u P_{N_2} v) \Vert_{N_{n_3}(T)} &\lesssim N_1 \Vert P_{N_1} u P_{N_2} v \Vert_{L^1_{T(N_3)} L_x^2} \\
&\lesssim N_1 N_1^{\frac{a-2}{2}} \Vert P_{N_1} u P_{N_2} v \Vert_{L_{T(N_3)}^2 L_x^2} \\
&\lesssim N_2^{n-1/2} \Vert P_N u \Vert_{F_n} \Vert P_K v \Vert_{F_k}
\end{split}
\end{equation*}
Suppose that $N_1 \sim N_2 \sim N_3$ and $n=2$. Using duality we have
\begin{equation}
\label{eq:HighHighHighNonlinearity}
\Vert P_{N_1} \partial_{x_1} (P_{N_2} u P_{N_3} v) \Vert_{N_{n_1}} = \sup_{\Vert w \Vert_{V_0^2} = 1} \int \int P_{N_1} w \partial_{x_1} (P_{N_2} u P_{N_3} v) dx dt
\end{equation}
Now, we use Proposition \ref{prop:HighHighHighInteraction} to apply a bilinear Strichartz estimate on two factors, say $w$ and $u$, to find
\begin{equation*}
\begin{split}
\eqref{eq:HighHighHighNonlinearity} &\lesssim N_1 \sup_{w} \Vert P_{N_1} w P_{N_2} u \Vert_{L^2_{t,x}} \Vert P_{N_3} v \Vert_{L^2_{t,x}} \\
&\lesssim N_1 N_2^{\frac{1-a}{2}} \log \langle N_2 \rangle \sup_w \Vert P_{N_1} w \Vert_{V^2} \Vert P_{N_2} u \Vert_{V^2} N_3^{a/2-1} \Vert P_{N_3} v \Vert_{F_{n_3}}
\end{split}
\end{equation*}
which is sufficient.\\
For $n \geq 3$ we use two $L^4_{t,x}$-Strichartz estimates instead:
\begin{equation*}
\begin{split}
\Vert P_{N_3} \partial_{x_1} (P_{N_1} u P_{N_2} v) \Vert_{N_{n_3}(T)} &\lesssim N_3 \Vert P_{N_1} u P_{N_2} v \Vert_{L^1_{T(N_3)} L_x^2} \\
&\lesssim N_3 N_3^{\frac{a-2}{2}} N^{\frac{n-(a+1)}{2}} \Vert P_{N_1} u \Vert_{F_{n_1}} \Vert P_{N_2} v \Vert_{F_{n_2}} \\
&\lesssim N_3^{\frac{n-1}{2}} \Vert P_{N_1} u \Vert_{F_{n_1}} \Vert P_{N_2} v \Vert_{F_{n_2}},
\end{split}
\end{equation*}
which is again sufficient.\\
Finally, suppose that $N_3 \ll N_1 \sim N_2$ . Here, we have to add localization in time which amounts to a factor $(N_1/N_3)^{2-a}$. Again we use duality to write
\begin{equation*}
\begin{split}
\Vert P_{N_3} \partial_{x_1} (P_{N_1} u P_{N_2} v) \Vert_{N_{n_3}} &\lesssim N_3 (N_1/N_3)^{2-a} \sup_{w} \int \int P_{N_3} w P_{N_1} u P_{N_2} v dx dt \\
&\lesssim N_3 \left( \frac{N_1}{N_3} \right)^{2-a} \sup_w \Vert P_{N_3} w P_{N_1} u \Vert_{L^2_{t,x}} \Vert P_{N_2} v \Vert_{L^2_{t,x}} \\
&\lesssim N_3 \left( \frac{N_1}{N_3} \right)^{2-a} N_1^{\frac{a-2}{2}} \left( \frac{N_3^{n-1}}{N_1^a} \right)^{1/2} \log^2 \langle N_1 \rangle \Vert P_{N_1} u \Vert_{F_{n_1}} \Vert P_{N_2} v \Vert_{F_{n_2}} \\
&\lesssim (N_1/N_3)^{1-a} N_3^{\frac{n-1}{2}} \log^2 \langle N_1 \rangle \Vert P_{N_1} u \Vert_{F_{n_1}} \Vert P_{N_2} v \Vert_{F_{n_2}}
\end{split}
\end{equation*}
and again carrying out the summation is straight-forward for $s>(n-1)/2$.
\end{proof}
\section{Energy estimates}
\label{section:EnergyEstimates}
First, we turn to the energy estimate which will yield a priori estimates provided that $s>s_a$:
\begin{proposition}
Let $n \geq 2$, $1 \leq a <2$ and let $u$ be a smooth solution to \eqref{eq:fractionalBenjaminOnoEquation}. Then, we find the following estimate to hold
\begin{equation}
\label{eq:EnergyEstimatesSolutions}
\Vert u \Vert^2_{E^s(T)} \lesssim \Vert u_0 \Vert^2_{H^s} + T \Vert u \Vert_{F^s(T)}^3
\end{equation}
provided that $s>s_a$.
\end{proposition}
\begin{proof}
The fundamental theorem of calculus yields
\begin{equation*}
\Vert P_N u(t) \Vert^2_{L^2} = \Vert P_N u_0 \Vert^2_{L^2} + \int_{0}^t ds \int_{\R^n} dx P_{N} u P_{N} \partial_{x_1} (u^2)
\end{equation*}
The time integral we treat with Littlewood-Paley decompositions and analyze the possible interactions separately.\\
Suppose that $N_1 \sim N_3 \gg N_2$. Then integration by parts and a commutator estimate yields after localization in time to intervals of size $N_1^{2-a}$
\begin{equation*}
\begin{split}
\left| \int_{I} \int_{\R^n} P_{N_1} u \partial_{x_1} (P_{N_2} u P_{N_3} u) dx dt \right| &\lesssim N_2 T N_1^{2-a} \Vert P_{N_1} u P_{N_2} u \Vert_{L^2_{t,x}} \Vert P_{N_3} u \Vert_{L^2_{t,x}} \\
&\lesssim T N_2 N_1^{2-a} \left( \frac{N_2^{n-1}}{N_1^{a}} \right)^{1/2} N_1^{\frac{a-2}{2}} \prod_i \Vert P_{N_i} u \Vert_{F_{n_i}} \\
&\lesssim T N_2^{s_a} \left( \frac{N_2}{N_1} \right)^{a-1} \prod_{i} \Vert P_{N_i} u \Vert_{F_{n_i}}
\end{split}
\end{equation*}
In case $N_1 \lesssim N_2 \sim N_3$ there is no point to integrate by parts, but apart from that the estimate is concluded along the lines of the above argument.
\end{proof}
Next, we proof the energy estimates which will yield Lipschitz continuity in $L^2$ for initial data in $H^s$, $s>s_a$ and continuity of the data-to-solution mapping after invoking the Bona-Smith approximation.
\begin{proposition}
Let $n\geq 2$, $ 1 \leq a < 2$ and $u_1,u_2$ be two smooth solutions to \eqref{eq:fractionalBenjaminOnoEquation} and denote $v=u_1-u_2$. Then, we find the following estimate to hold
\begin{align}
\label{eq:LipschitzContinuityL2}
\Vert v \Vert^2_{E^0(T)} &\lesssim \Vert v(0) \Vert_{L^2}^2 + T \Vert v \Vert_{F^0(T)}^2 ( \Vert u_1 \Vert_{F^s(T)} + \Vert u_2 \Vert_{F^s(T)}) \\
\label{eq:ContinuityHs}
\Vert v \Vert^2_{E^s(T)} &\lesssim \Vert v(0) \Vert_{H^s}^2 + T \Vert v \Vert_{F^s(T)}^3 + T \Vert v \Vert_{F^s(T)}^2 \Vert u_2 \Vert_{F^s(T)} + T \Vert v \Vert_{F^0(T)} \Vert v \Vert_{F^s(T)} \Vert u_2 \Vert_{F^{2s}(T)}
\end{align}
provided that $s>s_a$.
\end{proposition}
\begin{proof}
Performing the same reductions like above we have to estimate
\begin{equation*}
\left| \int \int P_{N_1} v \partial_{x_1} (P_{N_2} u P_{N_3} v) dx dt \right|
\end{equation*}
for $N_1 \sim N_3 \gg N_2$, $N_1 \lesssim N_2 \sim N_3$ and $N_3 \lesssim N_1 \sim N_2$.\\
The first case can be dealt with like in the corresponding estimate for solutions because we can still integrate by parts.\\
The second case does not require integration by parts and thus can be estimated like above. Finally, for the case $N_1 \lesssim N_2 \sim N_3$ we estimate
\begin{equation*}
\begin{split}
&\lesssim N_1 T N_1^{2-a} \Vert P_{N_1} v P_{N_3} v \Vert_{L^2_{t,x}} \Vert P_{N_2} u \Vert_{L^2_{t,x}} \\
&\lesssim T N_1^{2-a} N_3^{\frac{n-1}{2}} \Vert P_{N_1} v \Vert_{F_{n_1}} \Vert P_{N_2} u \Vert_{F_{n_2}} \Vert P_{N_3} v \Vert_{F_{n_3}}
\end{split}
\end{equation*}
This yields \eqref{eq:LipschitzContinuityL2} after summation.\\
To prove \eqref{eq:ContinuityHs} one writes
\begin{equation*}
\partial_t v + \partial_{x_1} |D|^{a} v = v \partial_{x_1} v + \partial_{x_1} (u_2 v)
\end{equation*}
The first term has the same symmetries like the term we encountered when proving a priori estimates for solutions. For the second term the only new estimate one has to carry out (due to impossibility to integrate by parts) is
\begin{equation*}
\sum_{1 \leq K \lesssim N} N^{2s} \int \int P_N v \partial_{x_1} (P_N u_2 P_K v) dx dt \lesssim T \Vert v \Vert_{F^0(T)} \Vert v \Vert_{F^s(T)} \Vert u_2 \Vert_{F^{2s}(T)}
\end{equation*}
which follows by the above means.
\end{proof}
\section{Proof of Theorem \ref{thm:LocalWellposednessEuclideanSpace}}
We shall be brief because the concluding arguments are already standard (cf. \cite{IonescuKenigTataru2008}). Below fix $s>s_a$.\\
By rescaling we are reduced to consider sufficiently small initial data. Firstly, we only consider initial data $u_0 \in H^\infty(\R^n)$. The energy method yields existence of solutions in $C([0,T^*],H^s(\R^n))$ for $s>n/2+1$, where $\lim_{T \to T^*} \Vert u(t) \Vert_{H^{2s}} = \infty$.\\
In a first step, we prove a priori estimates from
\begin{equation*}
\left\{\begin{array}{cl}
\Vert u \Vert_{F^s(T)} &\lesssim \Vert u \Vert_{E^s(T)} + \Vert \partial_{x_1} (u^2) \Vert_{N^s(T)} \\
\Vert u \partial_{x_1} u \Vert_{N^s(T)} &\lesssim \Vert u \Vert_{F^s(T)}^2 \\
\Vert u \Vert^2_{E^s(T)} &\lesssim \Vert u_0 \Vert_{H^s}^2 + T \Vert u \Vert_{F^s(T)}^3 \end{array} \right.
\end{equation*}
for solutions to \eqref{eq:fractionalBenjaminOnoEquation} by a bootstrap argument for $s>\frac{n+3}{2}-a$.\\
The above set of estimates yields
\begin{equation*}
\Vert u \Vert^2_{F^s(T)} \lesssim \Vert u_0 \Vert_{H^s}^2 + \Vert u \Vert^4_{F^s(T)} + T \Vert u \Vert^3_{F^s(T)}
\end{equation*}
Next, we invoke continuity of $E^s(T)$ and
\begin{equation*}
\lim_{T \to 0} \Vert u \Vert_{E^s(T)} \lesssim \Vert u_0 \Vert_{H^s}, \quad \lim_{T \to 0} \Vert \partial_{x_1} (u^2) \Vert_{N^s(T)} =0
\end{equation*}
For details see e.g. \cite{KochTataru2007}.\\
Consequently, the above set of estimates yields
\begin{equation}
\label{eq:FsControlSolutions}
\Vert u \Vert_{F^s(1)} \lesssim \Vert u_0 \Vert_{H^s}
\end{equation}
provided that $\Vert u_0 \Vert_{H^s}$ is chosen sufficiently small.\\
For $s^\prime > s$ we have
\begin{equation*}
\left\{\begin{array}{cl}
\Vert u \Vert_{F^{s^\prime}(T)} &\lesssim \Vert u \Vert_{E^{s^\prime}(T)} + \Vert \partial_{x_1} (u^2) \Vert_{N^{s^\prime}(T)} \\
\Vert u \partial_{x_1} u \Vert_{N^{s^\prime}(T)} &\lesssim \Vert u \Vert_{F^{s^\prime}(T)} \Vert u \Vert_{F^s(T)} \\
\Vert u \Vert^2_{E^{s^\prime}(T)} &\lesssim \Vert u_0 \Vert_{H^{s^\prime}}^2 + T \Vert u \Vert_{F^{s^\prime}(T)}^2 \Vert u \Vert_{F^s(T)} \end{array} \right.
\end{equation*}
Together with \eqref{eq:FsControlSolutions} this implies
\begin{equation*}
\Vert u \Vert_{F^{s^\prime}(1)} \lesssim \Vert u_0 \Vert_{H^{s^\prime}} \text{ for } s^\prime > s.
\end{equation*}
This a priori estimate for higher regularities together with the blow-up alternative shows that $T^* \geq 1$ provided that $\Vert u_0 \Vert_{H^s}$ is chosen sufficiently small.\\
Next, we argue that the set of estimates
\begin{equation*}
\left\{\begin{array}{cl}
\Vert v \Vert_{F^0(T)} &\lesssim \Vert v \Vert_{E^0(T)} + \Vert \partial_{x_1} (v (u_1+u_2)) \Vert_{N^0(T)} \\
\Vert \partial_{x_1} (v u_i) \Vert_{N^0(T)} &\lesssim \Vert v \Vert_{F^0(T)} \Vert u_i \Vert_{F^s(T)} \\
\Vert u \Vert^2_{E^0(T)} &\lesssim \Vert v(0) \Vert_{L^2}^2 + T \Vert v \Vert_{F^0(T)}^2 ( \Vert u_1 \Vert_{F^s(T)} + \Vert u_2 \Vert_{F^s(T)}) \end{array} \right.
\end{equation*}
yield an a priori estimate for $v$ in $L^2$ in dependence of $\Vert u_i \Vert_{H^s}$ for $s>\frac{n+3}{2} - a$.\\
Finally, the set of estimates
\begin{equation*}
\left\{\begin{array}{cl}
\Vert v \Vert_{F^s(T)} &\lesssim \Vert v \Vert_{E^s(T)} + \Vert \partial_{x_1} (v (u_1 + u_2)) \Vert_{N^s(T)} \\
\Vert \partial_{x_1} (v u_i) \Vert_{N^s(T)} &\lesssim \Vert v \Vert_{F^s(T)} \Vert u_i \Vert_{F^s(T)} \\
\Vert v \Vert^2_{E^s(T)} &\lesssim \Vert v(0) \Vert_{H^s}^2 + T \Vert v \Vert_{F^s(T)}^3 + T \Vert v \Vert_{F^0(T)} \Vert v \Vert_{F^s(T)} \Vert u_2 \Vert_{F^{2s}(T)} \end{array} \right.
\end{equation*}
allows us to conclude continuous dependence on the initial data by the classical Bona-Smith approximation (cf. \cite{BonaSmith1975,IonescuKenigTataru2008}).\\
For this purpose, let $u_2$ be the solution associated to $P_{\leq N} u_0$ and $u_1$ be the solution associated to $u_0$.\\
Due to the difference of initial data consisting only of high frequencies, the gain from estimating $\Vert v \Vert_{F^0}$ compensates the loss from estimating
\begin{equation*}
\Vert u_2 \Vert_{F^{2s}} \lesssim \Vert P_{\leq N} u_0 \Vert_{H^{2s}} \lesssim N^s \Vert P_{\leq N} u_0 \Vert_{H^s}
\end{equation*}
The data-to-solution mapping $H^s \rightarrow C([0,T],H^s) \cap F^s(T)$, which can also be constructed by the above means, is continuous, but not uniformly continuous because the approximation depends on the distribution of the Sobolev energy along the high frequencies, i.e., $\Vert P_{\geq N} u_0 \Vert_{H^s}$.

\section{Periodic solutions to fractional Zakharov-Kuznetsov equations}
Below, the above considerations regarding shorttime nonlinear and energy estimates are extended to the fully periodic case. Firstly, the function spaces are introduced.
\subsection{Function spaces in the periodic case}
Here, shorttime $X^{s,b}$-spaces adapted to periodic solutions are used (cf. \cite{Zhang2016}) to overcome the derivative loss. We will be brief because the function spaces are defined completely analogous to \cite{Zhang2016} with the basic function space properties remaining valid.\\
The dispersion relation for the two-dimensional Zakharov-Kuznetsov equation we denote by
\begin{equation*}
\omega(\xi,\eta) = \xi^3 + \xi \eta^2
\end{equation*}
For $k \in \mathbb{N}$ let $I^x_k = \{ \xi \in \R | |\xi| \in [2^{k-1},2^k) \}$ denote dyadic ranges on the real line and $I_k = \{ (\xi,\eta) \in \R^2 | |(\xi,\eta)| \in [2^{k-1}, 2^k) \}$. By $P_k$ and $P_{k,x}$ we denote the corresponding frequency projectors, i.e.,
\begin{equation*}
\begin{split}
 \widehat { P_{k,x} u } (\xi,\eta) &= 1_{I_{k,x}}(\xi) \hat{u}(\xi,\eta) \quad u \in L^2(\T^2) \\
\widehat { P_{k} u } (\xi,\eta) &= 1_{I_k}(\xi,\eta) \hat{u}(\xi,\eta)
\end{split}
\end{equation*}
Most of the time it will be fine to work with sharp cutoffs though in Subsection \ref{subsection:PeriodicEnergyEstimates} we adapt to smooth cutoffs, which will be denoted by $\tilde{P}_{k,x}$ or $\tilde{P}_k$, respectively.\\
For a time $T_0 \in (0,1]$, let $k_0 \geq 0$ be the greatest integer $k$ such that $2^k < 1/T_0$. For $k \in \mathbb{N} \cup \{ 0 \}$ define the dyadic $X^{s,b}$-type normed spaces
\begin{equation*}
\begin{split}
X_k &= X_k(\mathbb{Z}^2 \times \R) = \{ f \in L^2(\mathbb{Z}^2 \times \R) \; | \; f \text{ is supported in } I_k \times \mathbb{Z} \times \R \text{ and } \\
\Vert f_k \Vert_{X_k} &= 2^{k_0/2} \Vert \eta_{\leq k_0}(\tau - \omega(\xi,\eta)) f \Vert_{L^2_{\xi,\eta,\tau}} \\
&+ \sum_{j=k_0+1}^\infty 2^{j/2} \Vert \eta_j(\tau - \omega(\xi,\eta)) f \Vert_{L^2_{\xi,\eta,\tau}} < \infty \}
\end{split}
\end{equation*}
Recall the basic properties (\cite[Remark~2.1,~p.~259]{Zhang2016}):
For $f_k \in X_k$ we find the following estimate to hold:
\begin{equation*}
\Vert \int_{\R} |f_k(\xi,\eta,\tau)| d\tau \Vert_{L^2_{\xi,\eta}} \lesssim \Vert f_k \Vert_{X_k}
\end{equation*}
If $k,l \in \mathbb{N} \cap \{0 \}$, $l \geq k_0$ and $f_k \in X_k$, then
\begin{equation*}
\label{eq:LowModulationEstimate}
\begin{split}
&\sum_{j=l+1}^\infty 2^{j/2} \Vert \eta_j(\tau - \omega(\xi,\eta)) \int |f_k(\xi,\eta,\tau^\prime) 2^{-l} (1+2^{-l}|\tau-\tau^\prime|)^{-4} d\tau^\prime \Vert_{L^2} \\
&+ 2^{l/2} \Vert \eta_{\leq l}(\tau - \omega(\xi,\eta)) \int |f_k(\xi,\eta,\tau^\prime) 2^{-l} (1+2^{-l}|\tau-\tau^\prime|)^{-4} d\tau^\prime \Vert_{L^2} \lesssim \Vert f_k \Vert_{X_k}
\end{split}
\end{equation*}
Consequently, for $f \in X_k$ we find for $l \geq k_0$, $t_0 \in \R$, $\gamma \in \mathcal{S}(\R)$
\begin{equation*}
\Vert \mathcal{F}_{t,x}[\gamma (2^l t - t_0) \mathcal{F}^{-1}_{t,x}(f)] \Vert_{X_k} \lesssim_{\gamma} \Vert f \Vert_{X_k}
\end{equation*}
The $X_k$-spaces relate to the space-time Fourier transform of the original functions after frequency localization. Let 
\begin{equation*}
\begin{split}
E_{k,x} &= \{ \phi: \T^2 \rightarrow \R | \; \hat{\phi}  \text{ supported in }  I_k \times \Z, \; \Vert \phi \Vert_{E_k} = \Vert \hat{\phi} \Vert_{\ell^2_{\xi,\eta}} < \infty \} \\
E_k &= \{ \phi: \T^2 \rightarrow \R | \; \hat{\phi}  \text{ supported in } \; A_k, \; \Vert \phi \Vert_{E_k} = \Vert \phi \Vert_{L^2} < \infty \}
\end{split}
\end{equation*}
Next, define
\begin{equation*}
\begin{split}
F_{k,x} &= \{ u_k \in C(\R;E_{k,x}) | \Vert u_k \Vert_{F_{k,x}} = \sup \Vert \mathcal{F}[ u_k \eta_0(2^k(t-t_k))] \Vert_{X_k} < \infty \} \\
N_{k,x} &= \{ u_k \in C(\R;E_{k,x}) | \Vert u_k \Vert_{N_{k,x}} =  \sup_{t_k \in \R} \Vert (\tau - \omega(\xi,\eta)+i2^k)^{-1} \mathcal{F}[ u_k \eta_0(2^k(t-t_k))] \Vert_{X_k} < \infty \}
\end{split}
\end{equation*}
and for $T \in (0,T_0]$, let
\begin{equation*}
\begin{split}
F_{k,x}(T) = \{ u_k \in C([-T,T]; E_{k,x}) | \Vert u_{k} \Vert_{F_{k,x}(T)} = \inf_{\substack{ \tilde{u}_k = u_k \text{ in }\\ \T^2 \times [-T,T] }} \Vert \tilde{u}_k \Vert_{F_{k,x}} < \infty \} \\
N_{k,x}(T) = \{ u_k \in C([-T,T]; E_{k,x}) | \Vert u_{k} \Vert_{N_{k,x}(T)} = \inf_{\substack{ \tilde{u}_k = u_k \text{ in }\\ \T^2 \times [-T,T] }} \Vert \tilde{u}_k \Vert_{N_{k,x}} < \infty \}
\end{split}
\end{equation*}
The spaces $F^s_x(T)$, $N_x^s(T)$, $E^s_x(T)$ are assembled by Littlewood-Paley theory. Let $C=C([-T,T],H^\infty_0(\T^2))$ and define
\begin{equation*}
\begin{split}
F^s_x(T) &= \{ u \in C | \Vert u \Vert^2_{F_x^s(T)} = \sum_{k \in \mathbb{N}_0} (2^{2sk} + 2^{2s k_0}) \Vert P_{k,x} u \Vert^2_{F_{k,x}(T)} < \infty \} \\
N^s_x(T) &= \{ u \in C | \Vert u \Vert^2_{N_x^s(T)} = \sum_{k \in \mathbb{N}_0} (2^{2sk} + 2^{2s k_0}) \Vert P_{k,x} u \Vert^2_{N_{k,x}(T)} < \infty \} \\
E^s_x(T) &= \{ u \in C | \Vert u \Vert^2_{E^s_x(T)} = \Vert P_{\leq k_0, x} u(0) \Vert^2_{H^s} + \sum_{k \geq k_0} \sup_{t_k \in [-T,T]} 2^{2sk} \Vert P_{k,x} u(t_k) \Vert^2_{E_k} < \infty \}
\end{split}
\end{equation*}
The isotropic pendant spaces $F_k,N_k,F_k(T),N_k(T),F^s(T),N^s(T),E^s(T)$ are defined mutatis mutandi, replacing the anisotropic frequency projector $P_{k,x}$ with $P_k$. The multiplier properties (cf. \cite[p.~260]{Zhang2016}) hold independent of the dispersion relation.

\subsection{Bilinear estimates}
\label{subsection:PeriodicBilinearEstimates}
Next, we derive bilinear convolution estimates for space-time Fourier transforms of functions localized in frequency and modulation.\\
Due to the lack of dispersion, Strichartz estimates like in Euclidean space are not available. Still the supports of the involved functions are estimated by properties of the resonance function. First, we focus on the two-dimensional Zakharov-Kuznetsov dispersion
\begin{equation}
\begin{split}
\Omega(\xi_1,\xi_2,\eta_1,\eta_2) &= \omega(\xi_1+\xi_2,\eta_1+\eta_2) - \omega(\xi_1,\eta_1) - \omega(\xi_2,\eta_2) \\
&= (\xi_1+\xi_2)^3 - \xi_1^3 - \xi_2^3 + (\xi_1+\xi_2)(\eta_1+\eta_2)^2 - \xi_1 \eta_1^2 - \xi_2 \eta_2^2 \\
&= 3(\xi_1+\xi_2)\xi_1 \xi_2 + \xi_1 \eta_2 (2\eta_1 + \eta_2) + \xi_2 \eta_1 (\eta_1 + 2 \eta_2)
\end{split}
\end{equation}
For $k,j \in \mathbb{N}_0$, let
\begin{equation*}
\begin{split}
D^x_{k,\leq j} &= \{ (\xi,\eta,\tau) | \xi \in I_k, \; | \tau - \omega(\xi,\eta) | \leq 2^j \} \\
D_{k, \leq j} &= \{ (\xi,\eta,\tau) | (\xi,\eta) \in A_k, \; | \tau - \omega(\xi,\eta) | \leq 2^j \}
\end{split}
\end{equation*}
First, we give an estimate in the anisotropic case:
\begin{lemma}
\label{lem:anisotropicEstimate}
Let $k_i,j_i \in \mathbb{N}$, $f_i : \mathbb{Z}^2 \times \R \rightarrow \R_+, \; f_i \in L^2, supp(f_i) \subseteq D^x_{k_i,\leq j_i}$ for $ i \in \{1,2,3\}$.\\
Then, we find the following estimate to hold:
\begin{equation}
\label{eq:anisotropicEstimate}
\int_{\Z^2 \times \R} (f_1 * f_2 ) f_3 \lesssim \Vert f_1 \Vert_{L^2} 2^{j_2/2} 2^{k_2/2} \Vert f_2 \Vert_{L^2} (1+2^{\frac{j_3-k_1}{4}}) \Vert f_3 \Vert_{L^2}
\end{equation}
\end{lemma}
\begin{proof}
Set $f_i^{\#}(\xi,\eta,\tau) = f_i(\xi,\eta,\tau + \omega(\xi,\eta))$ for $i \in \{1,2,3\}$ such that $\Vert f_i^{\#} \Vert_{L^2} = \Vert f_i \Vert_{L^2}$ and
\begin{equation}
\label{eq:anisotropicBilinearReduction}
\int_{\Z^2 \times \R} (f_1 * f_2) f_3 = \int d\tau_1 d\tau_2 \sum_{\substack{ \xi_1,\eta_1, \\ \xi_2, \eta_2}} f_1^{\#}(\xi_1,\eta_1,\tau_1) f_2^{\#}(\xi_2,\eta_2,\tau_2) f_3^{\#}(\xi_1+\xi_2,\eta_1+\eta_2,\tau_1+\tau_2+\Omega),
\end{equation}
where $\text{supp}(f_i^{\#}) \subseteq \{ (\xi,\eta,\tau) | | \xi | \sim 2^{k_i}, \; | \tau| \leq 2^{j_i} \}$.\\
Observe that $|\partial^2 \Omega / \partial \eta_2^2 | \sim |\xi_1 | \sim 2^{k_1}$. Consequently, an application of the Cauchy-Schwarz inequality in $\eta_2$ yields
\begin{equation*}
\begin{split}
\eqref{eq:anisotropicBilinearReduction} &\lesssim \int d\tau_1 d\tau_2 \sum_{\xi_1,\eta_1,\xi_2} f_1^{\#}(\xi_1,\eta_1,\tau_1) (1+2^{\frac{j_3-k_1}{4}}) \\
&\times \left( \sum_{\eta_2} | f_2^{\#}(\xi_2,\eta_2,\tau_2)|^2 |f_3^{\#}(\xi_1+\xi_2,\eta_1+\eta_2,\tau_1+\tau_2+\Omega)|^2 \right)^{1/2} \\
&\lesssim \Vert f_1^{\#} \Vert_{L^2} (1+2^{\frac{j_3-k_1}{4}}) \Vert f_3^{\#} \Vert_{L^2} \sum_{\xi_2} \int d\tau_2 \left( \sum_{\eta_2} |f_2^{\#}(\xi_2,\eta_2,\tau_2)|^2 \right)^{1/2} \\
&\lesssim \Vert f_1^{\#} \Vert_{L^2} 2^{k_2/2} 2^{j_2/2} \Vert f_2^{\#} \Vert_{L^2} (1+2^{\frac{j_3-k_1}{4}}) \Vert f_3^{\#} \Vert_{L^2},
\end{split}
\end{equation*}
where the penultimate estimate follows from Cauchy-Schwarz in $\tau_1$, $\eta_1$ and $\xi_1$ and the last line follows from Cauchy-Schwarz in $\tau_2$ and $\xi_2$.
\end{proof}
Next, we turn to the isotropic case:
\begin{lemma}
Let $k_i,j_i \in \mathbb{N}$, $f_i: \Z^2 \times \R \rightarrow \R_+$, $f_i \in L^2$, $supp(f_i) \subseteq D_{k_i, \leq j_i}$ for $i \in \{1,2,3\}$.
\begin{enumerate}
\item[(a)] Let $|k_1-k_3| \leq 5$, $k_2 \leq k_1 - 10$. Then, we find the following estimate to hold:
\begin{equation}
\label{eq:IsotropicEstimateZKHighLow}
\int_{\Z^2 \times \R} (f_1 * f_2) f_3 \lesssim \Vert f_1 \Vert_{L^2} 2^{j_2/2} 2^{k_2/2} \Vert f_2 \Vert_{L^2} (1+ 2^{\frac{j_3-2k_1}{2}} ) \Vert f_3 \Vert_{L^2}
\end{equation}
\item[(b)] Let $|k_1 - k_2| \leq 5, |k_2-k_3| \leq 5$ and $j_i \geq k_i$. Then, we find the following estimate to hold:
\begin{equation}
\label{eq:IsotropicEstimateZKHighHighHigh}
\int_{\Z^2 \times \R} (f_1 * f_2) f_3 \lesssim 2^{-k_1/2} \prod_{i=1}^3 2^{j_i/2} \Vert f_i \Vert_{L^2}
\end{equation}
\item[(c)] The estimate
\begin{equation}
\label{eq:CauchySchwarzLocalizedFunctions}
\int_{\Z^2 \times \R} (f_1 * f_2) f_3 \lesssim 2^{k_{\min}} 2^{j_{\min}/2} \prod_{i=1}^3 \Vert f_i \Vert_{L^2}
\end{equation}
holds true.
\end{enumerate}
\end{lemma}
\begin{proof}
(a): For the representation \eqref{eq:anisotropicBilinearReduction} we find
\begin{equation}
\label{eq:transversalityResonance}
\frac{\partial \Omega}{\partial \xi_2} = 3 (\xi_1+\xi_2)^2 + (\eta_1+\eta_2)^2 - 3 \xi_2^2 - \eta_2^2, \quad \left| \frac{\partial \Omega}{\partial \xi_2} \right| \sim 2^{2n_1}
\end{equation}
and an application of the Cauchy-Schwarz inequality in $\xi_2$ yields
\begin{equation*}
\begin{split}
&\lesssim \int d\tau_1 \tau_2 \sum_{\xi_1,\eta_1,\eta_2} f_1^{\#}(\xi_1,\eta_1,\tau_1) (1+2^{\frac{j_3-2n_1}{2}}) \\
&\times \left( \sum_{\xi_2} |f_2^{\#}(\xi_2,\eta_2,\tau_2)|^2 |f_3^{\#}(\xi_1+\xi_2,\eta_1+\eta_2,\tau_1+\tau_2+\Omega)|^2 \right)^{1/2} \\
&\lesssim \Vert f_1^{\#} \Vert_{L^2} \int d\tau_2 \sum_{\eta_2} \left( \sum_{\xi_2} |f_2^{\#}(\xi_2,\eta_2,\tau_2)|^2 \right)^{1/2} (1+2^{\frac{j_3-2n_1}{2}}) \Vert f_3^{\#} \Vert_{L^2} \\
&\lesssim \Vert f_1^{\#} \Vert_{L^2} 2^{j_2/2} 2^{n_2/2} (1+2^{\frac{j_3-2n_1}{2}}) \Vert f_2^{\#} \Vert_{L^2} \Vert f_3^{\#} \Vert_{L^2},
\end{split}
\end{equation*}
where the penultimate estimate follows from applications of Cauchy-Schwarz inequality in $\tau_1,\xi_1$ and $\eta_1$ and the last line from applications in $\eta_2$ and $\tau_2$.\\[0.5cm]
(b): In case there are $x_1$-frequencies comparable to $2^{k_1}$ the estimate \eqref{eq:anisotropicEstimate} is sufficient as the claim follows due to $j_i \geq k_i$.\\
Hence, we suppose next that the $x_1$-frequencies are much smaller than $2^{k_1}$ and the $x_2$-frequencies are comparable to $2^{k_1}$. Then, either $\eta_1$ or $\eta_2$ has the same sign like $\eta_1+\eta_2$. Suppose that it is $\eta_2$. Then, we estimate
\begin{equation*}
\begin{split}
&\int d\tau_1 \int d\tau_2 \sum_{\xi_1,\eta_1,\eta_2} f_1^{\#}(\xi_1,\eta_1,\tau_1) \sum_{\xi_2} f_2^{\#}(\xi_2,\eta_2,\tau_2) f_3^{\#}(\xi_1+\xi_2,\eta_1+\eta_2,\tau_1+\tau_2+\Omega) \\
&\lesssim \int d\tau_1 \int d\tau_2 \sum_{\xi_1,\eta_1,\eta_2} f_1^{\#}(\xi_1,\eta_1,\tau_1) (1+2^{\frac{j_3-2n_1}{2}}) \\
&\times \left( \sum_{\xi_2} |f_2^{\#}(\xi_2,\eta_2,\tau_2)|^2 |f_3^{\#} (\xi_1+\xi_2,\eta_1+\eta_2,\tau_1+\tau_2+\Omega)|^2 \right)^{1/2} \\
&\lesssim (1+2^{\frac{j_3-2n_1}{2}}) \Vert f_1^{\#} \Vert_{L^2} \int d\tau_2 \sum_{\eta_2} \left( |f_2^{\#}(\xi_2,\eta_2,\tau_2)|^2 \right)^{1/2} \Vert f_3^{\#} \Vert_{L^2} \\
&\lesssim \Vert f_1^{\#} \Vert_{L^2} 2^{k_2/2} 2^{j_2/2} \Vert f_2^{\#} \Vert_{L^2} (1+2^{\frac{j_3-2n_1}{2}}) \Vert f_3^{\#} \Vert_{L^2}
\end{split}
\end{equation*}
where in the first estimate we used $\partial \Omega/\partial \xi_2 = \eta_1(\eta_1+2\eta_2)$, $|\partial \Omega / \partial \xi_2| \sim 2^{2n_1}$ and the Cauchy-Schwarz inequality in $\xi_2$ and in the second line Cauchy-Schwarz in $\xi_1,\eta_1,\tau_1$ and finally in $\tau_2$ and $\eta_2$.\\
(c): \eqref{eq:CauchySchwarzLocalizedFunctions} follows from applications of Cauchy-Schwarz inequality without using the resonance function.
\end{proof}
The above estimates in the isotropic case extend to dispersion generalizations by virtue of the transversality considerations from Section \ref{section:BilinearStrichartzEstimates}.
Let 
\begin{equation*}
\begin{split}
\varphi_a(\xi,\eta) &= \xi (\xi^2 + \eta^2)^{a/2} \\
D^a_{k_i,\leq j_i} &= \{ (\xi,\eta,\tau) \in \mathbb{Z}^2 \times \R | |(\xi,\eta)| \sim 2^{k_i}, |\tau - \varphi_a(\xi,\eta)| \lesssim 2^{j_i} \}
\end{split}
\end{equation*}
The considerations from Section \ref{section:BilinearStrichartzEstimates} can be utilized in the following way:\\
Consider $u_i \in L^2(\T^2 \times \R)$, real-valued, with $f_i = \mathcal{F}_{t,x}[u_i]$, $supp(f_i) \subseteq D^a_{k_i,\leq j_i}$ and moreover,
\begin{equation*}
|\nabla \varphi_a(\xi_2) - \nabla \varphi_a(\xi_3)| \gtrsim V \quad \xi_i \in supp(f_i)
\end{equation*}
Suppose that $|\partial_{\xi} \varphi_a(\xi_1,\eta_1) - \partial_{\xi} \varphi_a(\xi_2,\eta_2)| \gtrsim V $. In case another partial derivative dominates the conclusion follows likewise.\\
It follows that
\begin{equation*}
\begin{split}
\int dt \int dx u_1 u_2 u_3 &= \int dt \int dx u_1 u_2 \overline{u_3} = \int d\tau \sum_{\xi} (f_1 * f_2)(\xi,\tau) \overline{f_3(\tau,\xi)} \\
&= \int d\tau_1 \int d\tau_2 \sum_{\xi_1,\xi_2} f_1(\xi_1,\tau_1) f_2(\xi_2,\tau_2) \tilde{f}_3(\xi_1+\xi_2,\eta_1+\eta_2,\tau_1+\tau_2) \\
&= \int d\tau_1 \int d\tau_2 \sum_{\xi_1,\xi_2} f_1^{\#}(\xi_1,\tau_1-\omega(\xi_1)) f_2^{\#}(\xi_2,\tau_2) \\
&\times \tilde{f}_3(\xi_1+\xi_2,\tau_1+\tau_2+\omega(\xi_1)+\omega(\xi_2)-\omega(\xi_1+\xi_2))
\end{split}
\end{equation*}
Next, use Cauchy-Schwarz inequality in $\xi_2$ which will give a factor $1+(2^{j_3}/V)^{1/2}$ and following along the above lines we find
\begin{equation}
\label{eq:periodicTransversality}
\left| \int dt \int_{\T^2} dx u_1 u_2 u_3 \right| \lesssim 2^{k_2/2} 2^{j_2/2} \Vert u_1 \Vert_{L^2} \Vert u_2 \Vert_{L^2} (1+(2^{j_3}/V)^{1/2}) \Vert u_3 \Vert_{L^2}
\end{equation}
We record the following isotropic estimates for dispersion relations of two dimensional generalized Benjamin-Ono equations.
\begin{lemma}
Let $k_i,j_i \in \mathbb{N}$, $i \in \{1,2,3\}$, $f_i: \mathbb{Z}^2 \times \R \rightarrow \R_+$, $f_i \in L^2$, $supp(f_i) \subseteq D^a_{k_i,\leq j_i}$ for $i \in \{1,2,3\}$.\\
\begin{enumerate}
\item[(a)] If $|k_1-k_3| \leq 5$, $k_2 \leq k_1-10$ and $j_3 \geq k_3$, then
\begin{equation}
\label{eq:PeriodicHDBOHighLowHigh}
\int_{\Z^2 \times \R} (f_1 * f_2) f_3 \lesssim \Vert f_1 \Vert_{L^2} 2^{j_2/2} 2^{k_2/2} \Vert f_2 \Vert_{L^2} 2^{\frac{j_3-k_3}{2}} \Vert f_3 \Vert_{L^2}
\end{equation}
\item[(b)] If $|k_1-k_2| \leq 5$, $|k_2-k_3| \leq 5$ and $ j_i \geq k_i$, then
\begin{equation}
\label{eq:PeriodicHDBOHighHighHigh}
\int_{\Z^2 \times \R} (f_1 * f_2) f_3 \lesssim 2^{-k_1/2} \prod_{i=1}^3 2^{j_i/2} \Vert f_i \Vert_{L^2}
\end{equation}
\end{enumerate}
\end{lemma}
\begin{proof}
The proofs of (a) and (b) are consequences of \eqref{eq:periodicTransversality}. For (a) observe that $|\nabla \varphi_a(\xi_2) - \nabla \varphi_a(\xi_3)| \sim 2^{a k_1}$ and for (b) use the symmetry in $f_1$, $f_2$ and $f_3$ and Proposition \ref{prop:HighHighHighInteraction}.
\end{proof}
Recall that in higher dimensions we have the following transversality in case of separated frequencies: Let
\begin{equation}
\label{eq:dispersionRelation}
\varphi_a(\xi) = \xi_1 |\xi|^{a}, \quad \xi \in \R^n
\end{equation}
Suppose that $k_1 \leq k_2 - 10$, $\xi_i \in A_{k_i}$ for $i=1,2$. Then
\begin{equation}
\label{eq:TransversalityHigherDimensions}
| \partial_1 \varphi_a(\xi_1) - \partial_1 \varphi_a(\xi_2) | \gtrsim 2^{ak_1}
\end{equation}
For the Zakharov-Kuznetsov dispersion, i.e., $a=2$ we can also prove a certain transversality for frequencies of comparable size:
\begin{lemma}
\label{lem:ZKTransversalityPeriodicCase}
Let $\xi_i \in \Z^n$, $i=1,2,3$ and $n \geq 3$.\\
Assume $\xi_1 + \xi_2 + \xi_3 = 0$ and $|\xi_i| \sim 2^k$ and $1 \leq |\xi_{i1}| \ll 2^k$ for $i=1,2,3$. Then there are $i,j \in \{1,2,3\}$ such that
\begin{equation}
\label{eq:ZKTransversality}
|\nabla \varphi_2(\xi_i) - \nabla \varphi_2(\xi_j)| \gtrsim 2^k
\end{equation}
\end{lemma}
\begin{proof}
There must be one coordinate of size $2^k$, say $|\xi_{12}| \sim 2^k$. By this choice we use the symmetry of $\varphi_2$ in $\xi_2, \ldots , \xi_n$.\\
By convolution constraint there is another frequency with second component of size $2^k$, that is $|\xi_{22}| \sim 2^k$.\\
Suppose that $|\xi_{32}| \ll 2^k$.\\
In this case by convolution constraint we find $|\xi_{i1}| \gtrsim |\xi_{31}|$ for an $ i \in \{1,2\}$ and easily
\begin{equation*}
|\partial_2 \varphi_2(\xi_i) - \partial_2 \varphi_2(\xi_3)| = 2 | \xi_{i1} \xi_{i2} - \xi_{31} \xi_{32}| \gtrsim |\xi_{i1} \xi_{i2}| \sim 2^k
\end{equation*}
Next, suppose that $|\xi_{32}| \sim 2^k$. Here, we argue with the signs of the involved frequencies.\\
Suppose that $sgn(\xi_{12}) = sgn(\xi_{22}) = 1$, $sgn(\xi_{32}) = -1$.\\
In case $sgn(\xi_{11}) \neq sgn(\xi_{21})$ we have
\begin{equation*}
\begin{split}
|\partial_2 \varphi_2(\xi_1) - \partial_2 \varphi_2(\xi_2)| &= 2| \xi_{11} \xi_{12} - \xi_{21} \xi_{22}| \gtrsim |\xi_{11} \xi_{12}| \\
&\gtrsim 2^k
\end{split}
\end{equation*}
Thus, suppose that $sgn(\xi_{11}) = sgn(\xi_{21}) = 1$ or $sgn(\xi_{11}) = sgn(\xi_{21}) = -1$.\\
Then,
\begin{equation*}
\begin{split}
|\partial_2 \varphi_2(\xi_1) - \partial_2 \varphi_2(\xi_3)| &= 2|\xi_{11} \xi_{12} - \xi_{31} \xi_{32}| \\
&= 2|\xi_{11} \xi_{12} - (\xi_{11} + \xi_{21})(\xi_{12}+\xi_{22})| \\
&= 2|\xi_{21} \xi_{12} + \xi_{11} \xi_{22} + \xi_{21} \xi_{22}| \gtrsim |\xi_{21}| |\xi_{12}| \gtrsim 2^k
\end{split}
\end{equation*}
because all terms have the same sign in the above sum.
\end{proof}
This argument is insufficient to prove transversality for the higher dimensional fractional Zakharov-Kuznetsov dispersion relations in case of comparable frequencies.\\
Thus, we have no non-trivial estimates for the $High \times High \rightarrow High$-interaction for fractional Zakharov-Kuznetsov equations in higher dimensions. It might be possible though to prove an $L^4_{t,x}$-Strichartz estimate by decoupling (cf. \cite{BourgainDemeter2015,BourgainDemeter2017GeneralDecoupling}) adapting the argument from \cite{rsc2019StrichartzEstimatesDecoupling}.
\subsection{Nonlinear estimates}
\begin{proposition}
\label{prop:nonlinearEstimateZK2d}
Let $T \in (0,T_0]$. We find the following estimates to hold
\begin{align}
\label{eq:nonlinearEstimateZKI}
\Vert \partial_{x_1} (u v) \Vert_{N_x^{s^\prime}(T)} \lesssim T_0^{1/4} \Vert u \Vert_{F^s_x(T)} \Vert v \Vert_{F^{s^\prime}_x(T)} \\
\label{eq:nonlinearEstimateZKII}
\Vert \partial_{x_1} (u v) \Vert_{N_x^0(T)} \lesssim T_0^{1/4} \Vert u \Vert_{F^0_x(T)} \Vert v \Vert_{F^s_x(T)}
\end{align}
provided that $1 < s \leq s^\prime$. Moreover, estimates \eqref{eq:nonlinearEstimateZKI} and \eqref{eq:nonlinearEstimateZKII} also hold true, when replacing $N_x$ and $F_x$ by $N$ and $F$, respectively.
\end{proposition}
\begin{remark}
The argument below yields nonlinear estimates up to $H^{1/2}(\T^2)$. The regularity threshold $s>3/2$ comes from carrying out energy estimates.
\end{remark}
\begin{proof}
We prove the estimates in case of anisotropic frequency localization first. Choose $\tilde{u}, \tilde{v} \in C(\R, H^{3,0})$ such that
\begin{equation*}
\Vert P_{k,x} \tilde{u} \Vert_{F_{k,x}} \leq 2 \Vert P_{k,x} u \Vert_{F_{k,x}(T)} \text{ and } \Vert P_{k,x} \tilde{v} \Vert_{F_{k,x}} \leq 2 \Vert P_{k,x} v \Vert_{F_{k,x}(T)}
\end{equation*}
for $k \in \mathbb{N}$. Set $u_k = P_{k,x} \tilde{u} $ and $v_k = P_{k,x} \tilde{v}$. Then it suffices to consider the interactions
$High \times Low \rightarrow High$:
\begin{equation}
\label{eq:ShorttimeNonlinearZKHighLowHighInteraction}
\Vert P_{k,x} (\partial_{x_1} (u_{k_1} v_{k_2} ) \Vert_{N_{k,x}} \lesssim 2^{k_2/2} \Vert u_{k_1} \Vert_{F_{k_1,x}} \Vert v_{k_2} \Vert_{F_{k_2,x}} \quad (|k_1-k| \leq 5, k_2 \leq k-10) 
\end{equation}
$High \times High \rightarrow High$:
\begin{equation}
\label{eq:ShorttimeNonlinearZKHighHighHighInteraction}
\Vert P_{k,x} \partial_{x_1} (u_{k_1} v_{k_2}) \Vert_{N_{k,x}} \lesssim 2^{k/2} \Vert u_{k_1} \Vert_{F_{k_1,x}} \Vert v_{k_2} \Vert_{F_{k_2,x}} \quad (|k_1-k_2| \leq 5, |k_2-k_3| \leq 5)
\end{equation}
$High \times High \rightarrow Low$:
\begin{equation}
\label{eq:ShorttimeNonlinearZKHighHighLowInteraction}
\Vert P_{k,x} \partial_{x_1} (u_{k_1} v_{k_2}) \Vert_{N_{k,x}} \lesssim 2^{(k_1/2)+} \Vert u_{k_1} \Vert_{F_{k_1,x}} \Vert v_{k_2} \Vert_{F_{k_2,x}}
\end{equation}
In fact, the above estimates imply in case of $High \times Low \rightarrow High$- and $High \times High \rightarrow High$-interaction
\begin{equation*}
\Vert P_{k,x} \partial_{x_1} (u_{k_1} v_{k_2}) \Vert_{N_{k,x}} \lesssim 2^{-k_1/4} 2^{k_1} \Vert u_{k_1} \Vert_{F_{k_1,x}} \Vert v_{k_2} \Vert_{F_{k_2,x}}
\end{equation*}
and regarding $High \times High \rightarrow Low$-interaction
\begin{equation*}
(2^k + 2^{k_0}) \Vert P_{k,x} \partial_{x_1} (u_{k_1} v_{k_2}) \Vert_{N_{k,x}} \lesssim 2^{-k_1/4} (2^{k_1}+2^{k_0}) \Vert u_{k_1} \Vert_{F_{k_1,x}} 2^{k_2} \Vert v_{k_2} \Vert_{F_{k_2,x}}
\end{equation*}
Then the claim follows from the definition of the function spaces by summing over the frequencies.\\
We start with $High \times Low \rightarrow High$-interaction. By the definition of $N_{k,x}$ and $F_{k,x}$-spaces it suffices to show the estimate
\begin{equation}
\label{eq:ZKHighLowHighInteractionReduction}
2^k \sum_{j \geq k} 2^{-j/2} \Vert 1_{D^x_{k, \leq j}} (f_1 * f_2) \Vert_{L^2} \lesssim 2^{k_2/2} \prod_{i=1}^2 2^{j_i/2} \Vert f_i \Vert_{L^2}
\end{equation}
Here,
\begin{equation*}
f_i(\xi,\eta,\tau) = \begin{cases}
&\eta_{j_i}(\tau - \omega(\xi,\eta)) \mathcal{F}_{t,x}[u_i], \quad j_i > k \\
&\eta_{\leq j}(\tau - \omega(\xi,\eta)) \mathcal{F}_{t,x}[u_i], \quad j_i = k \end{cases}
\end{equation*}
To prove \eqref{eq:ZKHighLowHighInteractionReduction} we use duality to write
\begin{equation}
\label{eq:ShorttimeHighLowHighZKComputation}
\begin{split}
\Vert 1_{D^x_{k, \leq j}} (f_1 * f_2) \Vert_{L^2} &= \sup_{\Vert g_{k,j} \Vert_{L^2} = 1} \int_{\Z^2 \times \R} g_{k,j} (f_1 * f_2) \\
&\lesssim \sup_{\Vert g_{k,j} \Vert_{L^2} = 1} \Vert g_{k,j} \Vert_{L^2} 2^{k_2/2} 2^{j_2/2} \Vert f_2 \Vert_{L^2} 2^{\frac{j_1-k_1}{4}} \Vert f_1 \Vert_{L^2}\\
&\lesssim 2^{k_2/2} 2^{-k/2} \prod_{i=1}^2 2^{j_i/2} \Vert f_i \Vert_{L^2}
\end{split}
\end{equation}
where estimate \eqref{eq:anisotropicEstimate} was applied in the first step and the conclusion is due to $j_i \geq k$. Plugging \eqref{eq:ShorttimeHighLowHighZKComputation} into \eqref{eq:ZKHighLowHighInteractionReduction} yields \eqref{eq:ShorttimeNonlinearZKHighLowHighInteraction}.\\
The $High \times High \rightarrow High$-interaction is handled along the same lines.\\
In the case of $High \times High \rightarrow Low$-interaction we add further localization in time to length of $2^{-k_1}$ to estimate the resulting functions in $F_{k,x}$-spaces. Let $\gamma : \R \rightarrow [0,1]$ such that $\sum_{n \in \Z} \gamma^2(x-n) = 1 \quad \forall x \in \R$ and suppose that $k_1 \geq k_2$.\\
Then the lhs of \eqref{eq:ShorttimeNonlinearZKHighHighLowInteraction} is dominated by
\begin{equation*}
\begin{split}
&\sup_{t_k} \Vert (\tau - \omega(\xi,\eta) + i2^k)^{-1} 2^k 1_{I_k}(\xi) \sum_{|m| \leq C 2^{k_2-k}} \mathcal{F}[u_{k_1} \eta_0(2^{k}(t-t_k)) 
\gamma (2^{k_1}(t-t_k)-n) \\
&* \mathcal{F}[ v_{k_2} \eta_0(2^k(t-t_k)) \gamma( 2^{k_2}(t-t_k) - n))] \Vert_{X_k}
\end{split}
\end{equation*}
Thus, like above from the properties of the shorttime function spaces it suffices to prove
\begin{equation}
\label{eq:ZKHighHighLowInteractionReduction}
2^{k_1} \sum_{j \geq k} 2^{-j/2} \Vert 1_{D^x_{k,\leq j}} (f_1 * f_2) \Vert_{L^2} \lesssim 2^{(k_1/2)+} \prod_{i=1}^2 2^{j_i/2} \Vert f_i \Vert_{L^2},
\end{equation}
where $f_i: \Z^2 \times \R \rightarrow \R_+$ supported in $D_{k_i,\leq j_i}, j_i \geq k_1$, $i=1,2$.\\
The sum over $j$ we split into $k \leq j \leq 2k_1$ and $j > k_1$. For the first part we use duality and estimate \eqref{eq:anisotropicEstimate} like above to find
\begin{equation*}
\begin{split}
2^{k_1} \sum_{k \leq j \leq 2k_1} 2^{-j/2} \Vert 1_{D^x_{k, \leq j}} (f_1 * f_2) \Vert_{L^2} &\lesssim 2^{k_1} \sum_{k \leq j \leq 2k_1} 2^{-j/2} 2^{j/2} 2^{k/2} 2^{-k_1} \prod_{i=1}^2 2^{j_i/2} \Vert f_i \Vert_{L^2} \\
&\lesssim 2^{k/2} (2k_1-k) \prod_{i=1}^2 2^{j_i/2} \Vert f_i \Vert_{L^2}
\end{split}
\end{equation*}
In the second case we apply duality and estimate \eqref{eq:anisotropicEstimate} in another way to find
\begin{equation*}
\begin{split}
&2^{k_1} \sum_{j \geq 2k_1} 2^{-j/2} 2^{k_1/2} 2^{j_1/2} 2^{\frac{j_2-k}{4}} \Vert f_1 \Vert_{L^2} \Vert f_2 \Vert_{L^2} \\
&\lesssim 2^{\frac{k_1-k}{4}} \prod_{i=1}^2 2^{j_i/2} \Vert f_i \Vert_{L^2}, 
\end{split}
\end{equation*}
which is more than enough.\\
We turn to the estimates in the case of isotropic frequency localization. Again, we have to analyze the interactions from above. The pendant of \eqref{eq:ShorttimeNonlinearZKHighLowHighInteraction} reduces to
\begin{equation*}
2^k \sum_{j \geq k} 2^{-j/2} \Vert 1_{D_{k,\leq j}} (f_1 * f_2) \Vert_{L^2} \lesssim 2^{k_2/2} \prod_{i=1}^2 2^{j_i/2} \Vert f_i \Vert_{L^2},
\end{equation*}
where $supp(f_i) \subseteq D_{k_i,j_i}$, $j_i \geq k$.\\
To prove the above display use duality and apply estimate \eqref{eq:IsotropicEstimateZKHighLow} to find
\begin{equation*}
\Vert 1_{D_{k,\leq j}} (f_1 * f_2) \Vert_{L^2} = \sup_{\Vert g_{k,j} \Vert_{L^2} = 1} \int_{\Z^2 \times \R} g_{k,j} (f_1 * f_2) \lesssim 2^{k_2/2} 2^{-k/2} \prod_{i=1}^2 2^{j_i/2} \Vert f_{i} \Vert_{L^2}
\end{equation*}
For the $High \times High \rightarrow High$-interaction we split the sum over the output modulation variable into $n \leq j \leq 2n$ and $j \geq 2n$ to find
\begin{equation*}
\begin{split}
2^n \sum_{k \leq j \leq 2k} 2^{-j/2} \Vert 1_{D_{k,\leq j}} (f_1 * f_2) \Vert_{L^2} &\lesssim 2^n \sum_{k \leq j \leq 2k} 2^{-j/2} 2^{j/2} 2^{-k/2} \prod_{i=1^2} 2^{j_i/2} \Vert f_i \Vert_{L^2} \\
&\lesssim 2^{(n/2)+} \prod_{i=1}^2 2^{j_i/2} \Vert f_i \Vert_{L^2}
\end{split}
\end{equation*}
after applying duality and estimate \eqref{eq:IsotropicEstimateZKHighHighHigh}.\\
For the high modulation output apply duality and estimate \eqref{eq:CauchySchwarzLocalizedFunctions} to find
\begin{equation*}
\begin{split}
2^k \sum_{j \geq 2k} 2^{-j/2} \Vert 1_{D_{k,\leq j}} (f_1 * f_2) \Vert_{L^2} &\lesssim 2^n \sum_{j \geq 2n} 2^{-j/2} 2^{n/2} \prod_{i=1}^2 2^{j_i/2} \Vert f_i \Vert_{L^2} \\
&\lesssim 2^{n/2} \prod_{i=1}^2 2^{j_i/2} \Vert f_i \Vert_{L^2}
\end{split}
\end{equation*}
For $High \times High \rightarrow Low$-interaction we argue similarly: Taking into account the additional time localization it suffices to prove
\begin{equation}
2^{k_1} \sum_{j \geq k} 2^{-j/2} \Vert 1_{D_{k, \leq j}} (f_1 * f_2) \Vert_{L^2} \lesssim 2^{k_1/2} \prod_{i=1}^2 2^{j_i/2} \Vert f_i \Vert_{L^2},
\end{equation}
where $supp(f_i) \subseteq D_{k_i,\leq j_i}$, $j_i \geq j$ for $i=1,2$.\\
Again, the sum over $j$ is split into $k \leq j \leq 2k_1$, $j \geq 2k_1$. In the first case, we use duality and apply \eqref{eq:IsotropicEstimateZKHighLow} to find
\begin{equation*}
2^{k_1} \sum_{k \leq j \leq 2k_1} 2^{-j/2} 2^{j/2} 2^{k/2} 2^{-k_1/2} \prod_{i=1}^2 2^{j_i/2} \Vert f_i \Vert_{L^2} \lesssim (2k_1-k) 2^{k/2} \prod_{i=1}^2 2^{j_i/2} \Vert f_i \Vert_{L^2}
\end{equation*}
In the second case, estimate \eqref{eq:CauchySchwarzLocalizedFunctions} yields
\begin{equation}
\begin{split}
2^{k_1} \sum_{j \geq 2k_1} 2^{-j/2} \Vert 1_{D_{k,\leq j}} (f_1 * f_2) \Vert_{L^2} &\lesssim 2^{k_1} \sum_{j \geq 2k_1} 2^{-j/2} 2^k 2^{-k_1/2} \prod_{i=1}^2 2^{j_i/2} \Vert f_i \Vert_{L^2} \\
&\lesssim 2^{k_1/2} \prod_{i=1}^2 2^{j_i/2} \Vert f_i \Vert_{L^2}
\end{split}
\end{equation}
The proof is complete.
\end{proof}
\subsection{Energy estimates}
\label{subsection:PeriodicEnergyEstimates}
Purpose of this section is to propagate the energy norm of solutions and differences of solutions in terms of shorttime norms. We prove the following proposition:
\begin{proposition}
\label{prop:EnergyTransferPeriodicSolutions}
Let $T \in (0,1]$, $s>3/2$ and $u \in C([-T,T],H^\infty_0(\T^2))$ be a smooth solution to \eqref{eq:fractionalBenjaminOnoEquation} for $a=2$, $n=2$. Then we find the following estimate to hold:
\begin{equation}
\label{eq:energyEstimatePeriodicZK}
\Vert u \Vert^2_{E^s(T)} \lesssim \Vert u_0 \Vert_{H^s}^2 + T \Vert u \Vert_{F^s(T)}^3
\end{equation}
For two solutions to \eqref{eq:fractionalBenjaminOnoEquation} $u_i \in C([-T,T],H^\infty_0)$ with initial data $\phi_i$, $i=1,2$, the function $v=u_1-u_2$ satisfies the estimate
\begin{equation}
\label{eq:energyEstimatePeriodicZKDifferencesL2}
\Vert v \Vert^2_{E^0(T)} \lesssim \Vert v_1 - v_2 \Vert_{L^2}^2 + T \Vert v \Vert_{F^0(T)}^2 ( \Vert u_1 \Vert_{F^s(T)} + \Vert u_2 \Vert_{F^s(T)})
\end{equation}
and
\begin{equation}
\label{eq:energyEstimatePeriodicZKDifferencesHs}
\Vert v \Vert^2_{E^s(T)} \lesssim \Vert v_0 \Vert^2_{H^s} + T \Vert v \Vert^3_{F^s(T)} + T \Vert v \Vert_{F^0(T)} \Vert v \Vert_{F^s(T)} \Vert u_2 \Vert_{F^{2s}(T)}
\end{equation}
The above estimates also remain valid after replacing $E^s$, $F^s$ with $E^s_x$, $F^s_x$, respectively.
\end{proposition}
The proof will be carried out by estimating the energy transfer in the following way: Suppose that $u$ is a smooth solution to
\begin{equation}
\label{eq:ForcedZK}
\partial_t u + \partial_{x_1} \Delta u = v
\end{equation}
Then we find for the evolution of the $L^2$-norm of the frequencies
\begin{equation*}
\Vert P_k u(t_k) \Vert^2_{L^2} = \Vert P_k u(0) \Vert^2_{L^2} + 2 \int_{\T^2 \times [0,t_k]} P_k u \tilde{P}_k v dx dy dt
\end{equation*}
The key estimates are carried out in the following lemma:
\begin{lemma}
Let $T>0$, $u_i \in F_{k_i}(T)$, $i=1,2,3$. We find the following estimate to hold:
\begin{equation}
\label{eq:FrequencyLocalizedEnergyEstimateZKI}
\left| \int_{\T^2 \times [0,T]} u_1 u_2 u_3 dx dy dt \right| \lesssim T 2^{k_{\min}/2} \prod_{i=1}^3 \Vert u_i \Vert_{F_{k_i}(T)} 
\end{equation}
Suppose that $k_1 < k -10$. Then we find the following estimate to hold:
\begin{equation}
\label{eq:FrequencyLocalizedEnergyEstimateZKII}
\left| \int_{\T^2 \times [0,T] } \tilde{P}_k u \partial_{x_1} \tilde{P}_k( u \tilde{P}_{k_1} v) dx dy dt \right| \lesssim T 2^{3k_1/2} \sum_{|m_1-k_1| \leq 5} \Vert v \Vert_{F_{k_1}(T)} \sum_{|k^\prime - k | \leq 10} \Vert P_{k^\prime} u \Vert^2_{F_{k^\prime}(T)}
\end{equation}
Furthermore, estimates \eqref{eq:FrequencyLocalizedEnergyEstimateZKI} and \eqref{eq:FrequencyLocalizedEnergyEstimateZKII} hold true after replacing $\tilde{P}_k$ with $\tilde{P}_{k,x}$ and $F_{k_i}(T)$ with $F_{k_i,x}(T)$.
\end{lemma}
\begin{proof}
We start with the proof of the isotropic estimates. By symmetry we can assume that $k_1 \leq k_2 \leq k_3$. Let $\tilde{u}_i \in F_{k_i}$ with $\Vert \tilde{u}_i \Vert_{F_{k_i}} \leq 2 \Vert u_i \Vert_{F_{k_i}(T)}$, $i=1,2,3$ from the definitions.\\
The $\tilde{u}_i$ will be denoted by $u_i$ to lighten the notation. In order to estimate the functions in the shorttime function spaces time has to be localized according to the highest frequency. Let $\gamma: \R \rightarrow [0,1]$ be a smooth function supported in $[-1,1]$ with
\begin{equation*}
\sum_{n \in \Z} \gamma^3(x-n) \equiv 1 \quad \forall x \in \R
\end{equation*}
The lhs of \eqref{eq:FrequencyLocalizedEnergyEstimateZKI} is dominated by
\begin{equation}
\label{eq:FrequencyLocalizedEnergyEstimateReduction}
\begin{split}
&\sum_{|n| \leq CT_0 2^{k_3}} \left| \sum_{j_i \geq k_i} \int_{\Z^2 \times \R} \eta_{j_1}(\tau- \omega(\xi,\eta)) \mathcal{F}_{t,x}(u_1 \gamma(2^{k_3} t -n) 1_{[0,T]}(t) ) \right. \\
& \left. ( \eta_{j_2}(\tau - \omega(\xi,\eta)) \mathcal{F}[u_2 \gamma(2^{k_3} t -n)]) * (\eta_{j_3} \mathcal{F}[u_3 \gamma(2^{k_3} t -n)]) d\xi d\eta d\tau  \right| \\
&= \sum_{n \in A} (\ldots) + \sum_{n \in B} (\ldots),
\end{split}
\end{equation}
where
\begin{equation*}
\begin{split}
A &= \{ n \in \Z | \gamma(2^{k_3} \cdot -n) 1_{[0,T]} \neq \gamma(2^{k_3} \cdot - n) \},\\
B &= \{ n \in \Z | \gamma(2^{k_3} \cdot - n) 1_{[0,T]} = \gamma(2^{k_3} \cdot -n) \}
\end{split}
\end{equation*}
In \eqref{eq:FrequencyLocalizedEnergyEstimateReduction} read $\eta_{j_i} = \eta_{\leq j_i}$; it is sufficient to derive bounds for this modulation variable decomposition according to \eqref{eq:LowModulationEstimate}.\\
Apparently, $|A| \leq 10$, $|B| \leq C_0 T 2^{k_3}$. The main contribution of $B$ is handled first.\\
Denote
\begin{equation*}
f_i = \eta_{j_i}(\tau - \omega(\xi,\eta)) \mathcal{F}_{t,x}[ u_1 \gamma(2^{k_3} t - n) 1_{[0,T]}(t) ], \quad i =1,2,3
\end{equation*}
We do not distinguish between different values of $n$ because the following estimates are independent of $n$.\\
In case $k_1 \leq k_2 - 10$ an application of \eqref{eq:IsotropicEstimateZKHighLow} yields
\begin{equation*}
\begin{split}
\sum_{n \in B} (..) &\lesssim T 2^{k_3} \sum_{j_i \geq k_i} 2^{j_1/2} 2^{k_1/2} (1+2^{\frac{j_2-2k_3}{2}}) \prod_{i=1}^3 \Vert f_i \Vert_{L^2} \\
&\lesssim T 2^{k_1/2} \sum_{j_i \geq k_i} \prod_{i=1}^3 2^{j_i/2} \Vert f_i \Vert_{L^2}
\end{split}
\end{equation*}
because $|k_2-k_3| \leq 5$ and $j_3 \geq k_3$.\\
In case $|k_1 - k_2| \leq 5$, $|k_2-k_3| \leq 5$ an application of \eqref{eq:IsotropicEstimateZKHighHighHigh} gives
\begin{equation*}
\begin{split}
\sum_{n \in B} (..) &\lesssim T 2^{k_3} \sum_{j_i \geq k_i} 2^{-k_1/2} \prod_{i=1}^3 2^{j_i/2} \Vert f_i \Vert_{L^2} \\
&\lesssim T 2^{k_3/2} \prod_{i=1}^3 \sum_{j_i \geq k_i} 2^{j_i/2} \Vert f_i \Vert_{L^2}
\end{split}
\end{equation*}
For the boundary terms note that sharp cutoffs in time are almost bounded in $X_k$, that is for an interval $I \subseteq \R$, $k \in \mathbb{N}_0$, $f_k \in X_k$ and $f_k^I = \mathcal{F}(1_I(t) \mathcal{F}^{-1}(f_k))$ (cf. \cite[p.~267]{Zhang2016})
\begin{equation*}
\sup_{j \in \mathbb{N}} 2^{j/2} \Vert \eta_j(\tau - \omega(\xi,\eta)) f_k^I \Vert_{L^2} \lesssim \Vert f_k \Vert_{X_k}
\end{equation*}
An application of Cauchy-Schwarz yields
\begin{equation*}
\begin{split}
\sum_{n \in B} (..) &\lesssim \sum_{j_i \geq k_i} 2^{j_1/2} 2^{k_1} \prod_{i=1}^3 \Vert f_i \Vert_{L^2} \\
&\lesssim 2^{k_1} 2^{-k_3/2} \prod_{i=1}^2 \sum_{j_i \geq k_i} 2^{j_i/2} \Vert f_i \Vert_{L^2} \sup_{j \in \mathbb{N}} 2^{j/2} \Vert f_j \Vert_{L^2}
\end{split}
\end{equation*}
which yields the claim.\\
For the proof of \eqref{eq:FrequencyLocalizedEnergyEstimateZKII} we integrate by parts (cf. \cite{IonescuKenigTataru2008}) to find
\begin{equation*}
\begin{split}
&\left| \int_{\T^2 \times [0,T]} \tilde{P}_k u \tilde{P}_k (\partial_{x_1} u \tilde{P}_{k_1} v) dx dy dt \right| \\
&\leq \left| \int_{\T^2 \times [0,T]} \tilde{P}_k u \tilde{P}_k (\partial_{x_1} u) \tilde{P}_{k_1} v dx dy dt \right| + C \sum_{i=1}^2 \left| \int_{\Z^2 \times \R} \mathcal{F}(\tilde{P}_k u)(\xi,\eta,\tau) \right. \\
&\times \left. \int_{\Z^2 \times \R} \mathcal{F}(\tilde{P}_{k_1} \partial_{x_i} v)(\xi_1,\eta_1,\tau_1) \mathcal{F} v(\xi-\xi_1,\eta-\eta_1,\tau-\tau_1) \psi_i(\xi,\xi_1,\eta,\eta_1) d\xi_{1} d\eta_1 d\tau_1 d\xi d\eta d\tau \right|
\end{split}
\end{equation*}
where $\psi_{i}$, $i=1,2$ are bounded and regular multipliers. The resulting expressions can be handled by \eqref{eq:IsotropicEstimateZKHighLow}.\\
In the anisotropic case we use similar arguments, but use \eqref{eq:anisotropicEstimate} instead to conclude \eqref{eq:FrequencyLocalizedEnergyEstimateZKI} and the commutator estimate for \eqref{eq:FrequencyLocalizedEnergyEstimateZKII} is actually easier because there are no derivatives in $x_2$-direction involved.
\end{proof}
We are ready to prove Proposition \ref{prop:EnergyTransferPeriodicSolutions}.
\begin{proof}[Proof of Proposition \ref{prop:EnergyTransferPeriodicSolutions}]
Following the remark after Proposition \ref{prop:EnergyTransferPeriodicSolutions} we find for a solution to \eqref{eq:fractionalBenjaminOnoEquation} 
\begin{equation*}
\Vert \tilde{P}_k u(t_k) \Vert_{L^2}^2 = \Vert \tilde{P}_k u(0) \Vert^2_{L^2} + 2 \int_0^T ds \int_{\T^2} dx dy \tilde{P}_k u \tilde{P}_k ( \partial_{x_1} u^2)  
\end{equation*}
For the integral we consider the following interactions:
$High \times Low \rightarrow High$:
\begin{equation}
\label{eq:FrequencyLocalizedEnergyEstimateHighLowHigh}
\int_0^T ds \int_{\T^2} dx dy \tilde{P}_k u \tilde{P}_k (\partial_{x_1} u \tilde{P}_{k_1}) \quad (k_1 \leq k-10)
\end{equation}
$High \times High \rightarrow High$:
\begin{equation}
\label{eq:FrequencyLocalizedEnergyEstimateHighHighHigh}
\int_0^T ds \int_{\T^2} dx dy \tilde{P}_k u \tilde{P}_k (\partial_{x_1} u \tilde{P}_{k_1} u) \quad (|k-k_1| \leq 5)
\end{equation}
$High \times High \rightarrow Low$:
\begin{equation}
\label{eq:FrequencyLocalizedEnergyEstimateHighHighLow}
\int_{0}^T ds \int_{\T^2} dx dy \tilde{P}_k u \partial_{x_1} (\tilde{P}_{k_1} u \tilde{P}_{k_2} u) \quad ( k \leq k_1-10, \; |k_1-k_2| \leq 5)
\end{equation}
$High \times Low \rightarrow High$-interaction is estimated by \eqref{eq:FrequencyLocalizedEnergyEstimateZKII} to
\begin{equation*}
\eqref{eq:FrequencyLocalizedEnergyEstimateHighLowHigh} \lesssim T 2^{3k_1/2} \sum_{|m-k| \leq 5} \Vert P_m u \Vert^2_{F_m(T)} \sum_{|m_1-k_1| \leq 5} \Vert P_{m_1} u \Vert_{F_{m_1}(T)}
\end{equation*}
and summing over $k_1 \leq k-10$ and square summing over $k$ gives \eqref{eq:energyEstimatePeriodicZK}.\\
In case of $High \times High \rightarrow High$-interaction estimate \eqref{eq:FrequencyLocalizedEnergyEstimateZKI} is used to obtain
\begin{equation*}
\eqref{eq:FrequencyLocalizedEnergyEstimateHighHighHigh} \lesssim T 2^{3k/2} \sum_{|m-k| \leq 5} \Vert P_m u \Vert^2_{F_m(T)} \sum_{|m_1-k_1| \leq 10} \Vert P_{m_1} u \Vert_{F_{m_1}(T)}
\end{equation*}
and square summing over $k$ gives \eqref{eq:energyEstimatePeriodicZK}.\\
$High \times High \rightarrow Low$-interaction is handled similarly; again, there is no point in rearranging the derivative.\\
To prove \eqref{eq:FrequencyLocalizedEnergyEstimateZKII} we write
\begin{equation*}
\begin{split}
\Vert \tilde{P}_k v(t_k) \Vert^2_{L^2} = \Vert \tilde{P}_k v(t_k) \Vert^2_{L^2} &= \Vert \tilde{P}_k(u_1-u_2)(0) \Vert^2_{L^2} \\
&+ 2 \int_0^T ds \int_{\T^2} dx dy \tilde{P}_k v \tilde{P}_k \partial_{x_1} (v(u_1+u_2))
\end{split}
\end{equation*}
and estimate $High \times High \rightarrow High$-interaction and $High \times High \rightarrow Low$-interaction like above to obtain \eqref{eq:FrequencyLocalizedEnergyEstimateZKII}. In case of $High \times Low \rightarrow High$-interaction one finds two different terms:
\begin{equation}
\label{eq:EnergyEstimateHighLowHighDifferencesSolutions}
\int_0^T ds \int_{\T^2} dx dy \tilde{P}_k v \tilde{P}_k (\partial_{x_1} v \tilde{P}_{k_1} (u_1+u_2)) \quad (k_1 \leq k-10)
\end{equation}
and
\begin{equation}
\label{eq:EnergyEstimateHighLowHighDifferencesSolutionsII}
\int_0^T ds \int_{\T^2} dx dy \tilde{P}_k v \tilde{P}_k (\partial_{x_1} (u_1+u_2) \tilde{P}_{k_1} v) \quad (k_1 \leq k -10)
\end{equation}
\eqref{eq:EnergyEstimateHighLowHighDifferencesSolutions} is estimated along the above lines because we can integrate by parts to arrange the derivative on the smallest frequency.\\
For \eqref{eq:EnergyEstimateHighLowHighDifferencesSolutionsII} we use estimate \eqref{eq:FrequencyLocalizedEnergyEstimateZKI} instead to find
\begin{equation*}
\begin{split}
\eqref{eq:EnergyEstimateHighLowHighDifferencesSolutionsII} &\lesssim 2^k 2^{k_1/2} \sum_{|m-k| \leq 5} \Vert P_m v \Vert_{F_m(T)} \sum_{|m-k| \leq 5} (\Vert P_m u_1 \Vert_{F_m(T)} + \Vert P_m u_2 \Vert_{F_m(T)}) \\
&\times \sum_{|m_1-k_1| \leq 5} \Vert P_{m_1} v \Vert_{F_{m_1}(T)}
\end{split}
\end{equation*}
and square summing in $k$ and summing over $k_1 \leq k -10$ gives \eqref{eq:energyEstimatePeriodicZKDifferencesL2}.\\
To prove \eqref{eq:energyEstimatePeriodicZKDifferencesHs} the solution to the difference equation is rewritten as
\begin{equation*}
\partial_t v + \partial_{x_1} \Delta v = \partial_{x_1} (v^2) + \partial_{x_1} (v u_2)
\end{equation*}
When estimating $\Vert v \Vert_{E^s(T)}$ for $s>3/2$ the contribution of $\partial_{x_1} (v^2)$ can be handled like in the proof of \eqref{eq:energyEstimatePeriodicZK}, which gives
\begin{equation*}
\sum_k 2^{2ks} \int_0^T ds \int_{\T^2} dx dy \tilde{P}_k v \tilde{P}_k \partial_{x_1} (v^2) \lesssim T \Vert v \Vert^3_{F^s(T)}
\end{equation*}
The contribution of $\partial_{x_1}(v u_2)$ can be treated like in the proof of \eqref{eq:energyEstimatePeriodicZK} except for the interaction
\begin{equation*}
\int_0^T ds \int_{\T^2} dx dy \tilde{P}_k v \tilde{P}_k \partial_{x_1} (u_2 \tilde{P}_{k_1} v) \quad (k_1 \leq k -10)
\end{equation*}
because here we can not integrate by parts like above. Instead estimate \eqref{eq:FrequencyLocalizedEnergyEstimateZKI} and square summing in $k$ and summation in $k_1 \leq k -10$ gives
\begin{equation*}
\sum_{k, k_1 \leq k-10} 2^{2ks} \int_0^T ds \int_{\T^2 } dx dy \tilde{P}_k v \tilde{P}_k \partial_{x_1} (u_2 \tilde{P}_{k_1} v) \lesssim T \Vert v \Vert_{F^s(T)} \Vert u_2 \Vert_{F^{2s}(T)} \Vert v \Vert_{F^0(T)}
\end{equation*}
In the anisotropic case the same strategy applies after deploying the energy estimate for anisotropic frequency localization.
\end{proof}

\subsection{Proof of Theorem \ref{thm:LocalWellposednessPeriodicCase}}
\begin{proof}[Proof of Theorem \ref{thm:LocalWellposednessPeriodicCase}]
Fix $s>3/2$. Here, instead of rescaling to small initial values which are considered for large times like in the proof of Theorem \ref{thm:LocalWellposednessEuclideanSpace} we consider arbitrary initial data for small times as in \cite{Zhang2016}.\\
We only demonstrate the proof of a priori estimates for smooth initial values. The additionally required arguments to construct the data-to-solution mapping are like in the proof of Theorem \ref{thm:LocalWellposednessEuclideanSpace}. For $0<T \leq T_0$ we find for a smooth solution
\begin{equation*}
\left\{\begin{array}{cl}
\Vert u \Vert_{F^s(T)} &\lesssim \Vert u \Vert_{E^s(T)} + \Vert \partial_{x_1} (u^2) \Vert_{N^s(T)} \\
\Vert u \partial_{x_1} u \Vert_{N^s(T)} &\lesssim T_0^{1/4} \Vert u \Vert_{F^s(T)}^2 \\
\Vert u \Vert^2_{E^s(T)} &\lesssim \Vert u_0 \Vert_{H^s}^2 + T_0 \Vert u \Vert_{F^s(T)}^3 \end{array} \right.
\end{equation*}
This implies
\begin{equation*}
\Vert u \Vert^2_{F^s(T)} \lesssim \Vert u_0 \Vert^2_{H^s} + T_0^{1/4} \Vert u \Vert^2_{F^s(T)} + T_0 \Vert u \Vert^3_{F^s(T)}
\end{equation*}
and since
\begin{equation*}
\lim_{T \to 0} \Vert u \Vert_{E^s(T)} \lesssim \Vert u_0 \Vert_{H^s}, \quad \lim_{T \to 0} \Vert \partial_{x_1} (u^2) \Vert_{N^s(T)} = 0
\end{equation*}
as in \cite[Lemma~6.3,~p.~278]{Zhang2016}. After choosing $T_0 = T_0(\Vert u_0 \Vert_{H^s})$ one proves a priori estimates by a bootstrap argument.
\end{proof}

\bibliographystyle{amsxport}

\begin{thebibliography}{10}

\bibitem{BenArtziKochSaut2003}
Matania Ben-Artzi, Herbert Koch, and Jean-Claude Saut.
\newblock Dispersion estimates for third order equations in two dimensions.
\newblock {\em Comm. Partial Differential Equations}, 28(11-12):1943--1974,
  2003.

\bibitem{Benjamin1974}
T.~{Benjamin}.
\newblock {Internal waves of permanent form in fluids of great depth}.
\newblock {\em J. Fluid Mech.}, 29:559--562, 1967.

\bibitem{BonaSmith1975}
J.~L. Bona and R.~Smith.
\newblock The initial-value problem for the {K}orteweg-de {V}ries equation.
\newblock {\em Philos. Trans. Roy. Soc. London Ser. A}, 278(1287):555--601,
  1975.

\bibitem{Bourgain1993KPII}
J.~Bourgain.
\newblock On the {C}auchy problem for the {K}adomtsev-{P}etviashvili equation.
\newblock {\em Geom. Funct. Anal.}, 3(4):315--341, 1993.

\bibitem{Bourgain1998RefinementsStrichartzInequality}
J.~Bourgain.
\newblock Refinements of {S}trichartz' inequality and applications to
  {$2$}{D}-{NLS} with critical nonlinearity.
\newblock {\em Internat. Math. Res. Notices}, (5):253--283, 1998.

\bibitem{BourgainDemeter2015}
Jean Bourgain and Ciprian Demeter.
\newblock The proof of the {$l^2$} decoupling conjecture.
\newblock {\em Ann. of Math. (2)}, 182(1):351--389, 2015.

\bibitem{BourgainDemeter2017GeneralDecoupling}
Jean Bourgain and Ciprian Demeter.
\newblock Decouplings for curves and hypersurfaces with nonzero {G}aussian
  curvature.
\newblock {\em J. Anal. Math.}, 133:279--311, 2017.

\bibitem{BustamanteJimenezUrreaMejia2019}
Eddye {Bustamante}, Jos{\'e} {Jim{\'e}nez Urrea}, and Jorge {Mej{\'\i}a}.
\newblock {Periodic Cauchy Problem for one Two-dimensional Generalization of
  the Benjamin-Ono Equation in Sobolev Spaces of Low Regularity}.
\newblock {\em arXiv e-prints}, page arXiv:1901.06329, Jan 2019.

\bibitem{ChristHolmerTataru2012}
Michael Christ, Justin Holmer, and Daniel Tataru.
\newblock Low regularity a priori bounds for the modified {K}orteweg-de {V}ries
  equation.
\newblock {\em Lib. Math. (N.S.)}, 32(1):51--75, 2012.

\bibitem{GinibreVelo1979}
J.~Ginibre and G.~Velo.
\newblock On a class of nonlinear {S}chr\"{o}dinger equations. {I}. {T}he
  {C}auchy problem, general case.
\newblock {\em J. Funct. Anal.}, 32(1):1--32, 1979.

\bibitem{GruenrockHerr2014}
Axel Gr\"{u}nrock and Sebastian Herr.
\newblock The {F}ourier restriction norm method for the {Z}akharov-{K}uznetsov
  equation.
\newblock {\em Discrete Contin. Dyn. Syst.}, 34(5):2061--2068, 2014.

\bibitem{Guo2012DispersionGeneralizedBenjaminOno}
Zihua Guo.
\newblock Local well-posedness for dispersion generalized {B}enjamin-{O}no
  equations in {S}obolev spaces.
\newblock {\em J. Differential Equations}, 252(3):2053--2084, 2012.

\bibitem{GuoPengWangWang2011}
Zihua Guo, Lizhong Peng, Baoxiang Wang, and Yuzhao Wang.
\newblock Uniform well-posedness and inviscid limit for the
  {B}enjamin-{O}no-{B}urgers equation.
\newblock {\em Adv. Math.}, 228(2):647--677, 2011.

\bibitem{HadacHerrKoch2009}
Martin Hadac, Sebastian Herr, and Herbert Koch.
\newblock Well-posedness and scattering for the {KP}-{II} equation in a
  critical space.
\newblock {\em Ann. Inst. H. Poincar\'{e} Anal. Non Lin\'{e}aire},
  26(3):917--941, 2009.

\bibitem{HadacHerrKoch2009Erratum}
Martin Hadac, Sebastian Herr, and Herbert Koch.
\newblock Erratum to ``{W}ell-posedness and scattering for the {KP}-{II}
  equation in a critical space'' [{A}nn. {I}. {H}. {P}oincar\'{e}---{AN} 26 (3)
  (2009) 917--941] [mr2526409].
\newblock {\em Ann. Inst. H. Poincar\'{e} Anal. Non Lin\'{e}aire},
  27(3):971--972, 2010.

\bibitem{HerrIonescuKenigKoch2010}
Sebastian Herr, Alexandru~D. Ionescu, Carlos~E. Kenig, and Herbert Koch.
\newblock A para-differential renormalization technique for nonlinear
  dispersive equations.
\newblock {\em Comm. Partial Differential Equations}, 35(10):1827--1875, 2010.

\bibitem{IonescuKenigTataru2008}
A.~D. Ionescu, C.~E. Kenig, and D.~Tataru.
\newblock Global well-posedness of the {KP}-{I} initial-value problem in the
  energy space.
\newblock {\em Invent. Math.}, 173(2):265--304, 2008.

\bibitem{KeelTao1998}
Markus Keel and Terence Tao.
\newblock Endpoint {S}trichartz estimates.
\newblock {\em Amer. J. Math.}, 120(5):955--980, 1998.

\bibitem{KillipVisan2018}
Rowan {Killip} and Monica {Visan}.
\newblock {KdV is wellposed in $H^{-1}$}.
\newblock {\em arXiv e-prints}, page arXiv:1802.04851, February 2018.

\bibitem{KochTzvetkov2003}
H.~Koch and N.~Tzvetkov.
\newblock On the local well-posedness of the {B}enjamin-{O}no equation in
  {$H^s({\Bbb R})$}.
\newblock {\em Int. Math. Res. Not.}, (26):1449--1464, 2003.

\bibitem{KochTzvetkov2005}
H.~Koch and N.~Tzvetkov.
\newblock Nonlinear wave interactions for the {B}enjamin-{O}no equation.
\newblock {\em Int. Math. Res. Not.}, (30):1833--1847, 2005.

\bibitem{KochTataru2005}
Herbert Koch and Daniel Tataru.
\newblock Dispersive estimates for principally normal pseudodifferential
  operators.
\newblock {\em Comm. Pure Appl. Math.}, 58(2):217--284, 2005.

\bibitem{KochTataru2007}
Herbert Koch and Daniel Tataru.
\newblock A priori bounds for the 1{D} cubic {NLS} in negative {S}obolev
  spaces.
\newblock {\em Int. Math. Res. Not. IMRN}, (16):Art. ID rnm053, 36, 2007.

\bibitem{LinaresPantheeRobertTzvetkov2018}
Felipe {Linares}, Mahendra {Panthee}, Tristan {Robert}, and Nikolay {Tzvetkov}.
\newblock {On the periodic Zakharov-Kuznetsov equation}.
\newblock {\em arXiv e-prints}, page arXiv:1809.02027, Sep 2018.

\bibitem{LinaresRianoRogersWright2019}
Felipe {Linares}, Oscar~G. {Ria{\~n}o}, Keith~M. {Rogers}, and James {Wright}.
\newblock {On a higher dimensional version of the Benjamin--Ono equation}.
\newblock {\em arXiv e-prints}, page arXiv:1901.04817, January 2019.

\bibitem{Maris2002}
Mihai Mari\c{s}.
\newblock On the existence, regularity and decay of solitary waves to a
  generalized {B}enjamin-{O}no equation.
\newblock {\em Nonlinear Anal.}, 51(6):1073--1085, 2002.

\bibitem{MolinetSautTzvetkov2001}
L.~Molinet, J.~C. Saut, and N.~Tzvetkov.
\newblock Ill-posedness issues for the {B}enjamin-{O}no and related equations.
\newblock {\em SIAM J. Math. Anal.}, 33(4):982--988, 2001.

\bibitem{MolinetPilod2015}
Luc Molinet and Didier Pilod.
\newblock Bilinear {S}trichartz estimates for the {Z}akharov-{K}uznetsov
  equation and applications.
\newblock {\em Ann. Inst. H. Poincar\'{e} Anal. Non Lin\'{e}aire},
  32(2):347--371, 2015.

\bibitem{PelinovskyShrira1995}
D.~{Pelinovsky} and V.~{Shrira}.
\newblock {Collapse transformation for self-focusing solitary waves in
  boundary-layer type shear flows}.
\newblock {\em Physics Letters A}, 206(3):195--202, 1995.

\bibitem{RibaudVento2012}
Francis Ribaud and St\'{e}phane Vento.
\newblock Well-posedness results for the three-dimensional
  {Z}akharov-{K}uznetsov equation.
\newblock {\em SIAM J. Math. Anal.}, 44(4):2289--2304, 2012.

\bibitem{RibaudVento2017}
Francis Ribaud and St\'{e}phane Vento.
\newblock Local and global well-posedness results for the
  {B}enjamin-{O}no-{Z}akharov-{K}uznetsov equation.
\newblock {\em Discrete Contin. Dyn. Syst.}, 37(1):449--483, 2017.

\bibitem{Saut2018}
Jean-Claude {Saut}.
\newblock {Benjamin-Ono and Intermediate Long Wave equation : modeling, IST and
  PDE}.
\newblock {\em arXiv e-prints}, page arXiv:1811.08652, November 2018.

\bibitem{rsc2018BilinearStrichartzEstimates}
R.~Schippa.
\newblock {On shorttime bilinear Strichartz estimates and applications to
  improve the energy method}.
\newblock {\em arXiv e-prints}, page arXiv:1810.04406, October 2018.

\bibitem{rsc2019StrichartzEstimatesDecoupling}
Robert {Schippa}.
\newblock {On Strichartz estimates from decoupling and applications}.
\newblock {\em arXiv e-prints}, page arXiv:1901.01177, Jan 2019.

\bibitem{Tomas1975}
Peter~A. Tomas.
\newblock A restriction theorem for the {F}ourier transform.
\newblock {\em Bull. Amer. Math. Soc.}, 81:477--478, 1975.

\bibitem{ZakharovKuznetsov1974}
V.~{Zakharov} and E.~{Kuznetsov}.
\newblock {On three dimensional solitons}.
\newblock {\em J. Exp. Theor. Phys.}, 39:285--286, 1974.

\bibitem{Zhang2016}
Yu~Zhang.
\newblock Local well-posedness of {KP}-{I} initial value problem on torus in
  the {B}esov space.
\newblock {\em Comm. Partial Differential Equations}, 41(2):256--281, 2016.

\end{thebibliography}
\end{document}